\documentclass[11pt]{article}

\usepackage{graphicx}
\usepackage{amssymb}
\usepackage{amsmath}
\usepackage{caption}
\usepackage{subcaption}
\usepackage{subfloat}
\usepackage{bbold}

\newtheorem{theorem}{\bf Theorem}[section]
\newtheorem{proposition}{\bf Proposition}[section]

\newtheorem{remark}{\bf Remark}[section]

\newtheorem{example}{\bf Example }[section]

\usepackage{color}

\newenvironment{proof}{
\begin{trivlist}
\item[\hspace{\labelsep}{\bf\noindent Proof. }] }{\par\hfill\end{trivlist}
\par}

\date{\empty}

\title{
\huge\bf Analysis and applications of the residual varentropy of random lifetimes\thanks{
To appear on {\bf Probability in the Engineering and Informational Sciences}. 
}}
\author{
Antonio  {\bf Di Crescenzo}\footnote{Corresponding author \ -- \ Address: 
Dipartimento di Matematica, Universit\`a degli Studi di Salerno, Via Giovanni Paolo II n.\ 132, I-84084 Fisciano (SA), Italy \ -- \ Email: adicrescenzo@unisa.it \ -- \ ORCID: 0000-0003-4751-7341}
 \quad  and  \quad  
Luca {\bf Paolillo}\footnote{Address: 
Dipartimento di Matematica, Universit\`a degli Studi di Salerno, Via Giovanni Paolo II n.\ 132, I-84084 Fisciano (SA), Italy \ -- \ Email: lpaolillo@unisa.it \ -- \ ORCID: 0000-0001-7146-4863} 
}

\begin{document}
 
\maketitle


\begin{abstract}
In  reliability theory and survival analysis, the residual entropy is known 
as a measure suitable to describe the dynamic information content in stochastic 
systems conditional on survival. Aiming to analyze the variability of such information content, 
in this paper we introduce the variance of the residual lifetimes, ``residual varentropy'' in short. 
After a theoretical investigation of some  properties of the residual varentropy, 
we illustrate certain applications related to the proportional hazards model and the 
first-passage times of an Ornstein-Uhlenbeck jump-diffusion process. 
\end{abstract}



\section{Introduction}
The differential entropy is a well-known information measure that represents the expectation of the information 
content of an absolutely continuous random variable. The corresponding variance is termed varentropy 
and is used in various applications of information theory, such as for the estimation of the performance 
of optimal block-coding schemes. Recent contributions on the varentropy can be found in various 
papers by Arikan \cite{Arikan2016}, Bobkov and Madiman \cite{Bobkov2011}, 
Fradelizi {\em et al.}\ \cite{Fradelizi2016}, 
Kontoyiannis and Verd\'u \cite{Kontoyannis2013}, 
\cite{Kontoyannis2014}, \cite{Verdu2012}. Most of such results have been aimed to mathematical 
properties or to applications in information theory. However, it should be pointed out that such information 
measures often deserve interest in other fields, such as  reliability and survival analysis. See, for instance,  
Nanda and Chowdhury \cite{nanda2019} for a recent comprehensive review on the Shannon's entropy 
and its applications in various fields. Several investigations have been oriented in the past to assess the information 
content of stochastic systems  with special attention to dynamic measures related to the residual lifetime, 
the past lifetime, the inactivity time and their suitable generalizations. However, no efforts have been 
dedicated to the analysis of the variance of the information content in dynamic contexts. 
\par
On the ground of the above remarks, the motivation of this paper is 
to investigate the varentropy of residual lifetimes 
in a field related to reliability theory. The main aim is to measure the variability of the dynamic 
information content of stochastic systems that are conditioned on survival. This investigation is  
motivated by the need of constructing new mathematical tools suitable to describe the time course 
of the information content in addition to the residual entropy.
Our attention is devoted to disclose properties of the varentropy of residual lifetimes. We give 
special attention to the conditions such that it is constant. We also discuss the effect of linear 
transformations and provide suitable lower and upper bounds. Moreover, we focus on certain applications 
involving the proportional hazards model and the 
first-passage times for  Ornstein-Uhlenbeck jump-diffusion processes. 
\par
The paper is organized as follows: 
In Section 2, we recall some basic results on useful notions of information theory and reliability theory, 
with special attention to the varentropy and the residual lifetimes. 
In Section 3, we introduce the residual varentropy and investigate some properties of such new measure. 
Among other facts, we find conditions involving the generalized hazard rate such that the residual varentropy 
is constant, we discuss the effect of linear transformations, and obtain suitable upper and lower bounds 
for the residual varentropy.  
Section 4 is devoted to some applications. We first deal with the proportional hazard rates model and 
the reliability analysis of series system. We also discuss an application to first-passage times of the 
Ornstein-Uhlenbeck jump-diffusion process arising from the Ehrenfest model subject to catastrophes. 
\par
Throughout the paper, $\mathbb{E}[\,\cdot\,]$ denotes expectation, $g'$ means the derivative of $g$, 
``$\log$'' is the natural logarithm, and we set $0\log 0 =0$ by convention. 
Moreover, notation $[X|B]$ is adopted for a random variable whose distribution is identical to that of $X$ conditional on $B$. 
\section{Background} 
Let $X$ be a random variable defined on a probability space $(\Omega, {\cal F}, \mathbb{P})$,  
and let $F(t)=\mathbb{P}(X\leq t)$, $t\in \mathbb{R}$, be its cumulative distribution function (cdf). 
We denote by $\overline{F}(t)=1-F(t)$ the complementary distribution function, also known as survival function. 

\subsection{Varentropy}
If $X$ is absolutely continuous with probability density function (pdf) $f(t)$, 
we can introduce the random variable
\begin{equation}\label{eq:variabile informazione}
 IC(X)=-\log f(X),
\end{equation}
that is often referred as the (random) information content of $X$. 
We recall that $IC(X)$ is the natural counterpart of the 
number of bits needed to represent $X$ in the discrete case by a coding scheme that minimizes 
the average code length (see  \cite{Shannon1948}). A very common uncertainty measure 
is the expectation of the information content of $X$, given by 
\begin{equation}\label{eq:shannonentropy}
 H(X):= \mathbb{E}[IC(X)]= - \mathbb{E}[\log f(X)]
 = -\int_{-\infty}^\infty{f(x)\log{f(x)}}\,{\rm d}x,
\end{equation}
which is termed 
differential entropy. Intuitively, $H(X)$ measures the expected uncertainty contained in $f(x)$ about the 
predictability of an outcome of $X$. We remark that $H(X)$ may or may not exist (in the Lebesgue sense). 
We remark that the differential entropy is also related to the evaluation of the size of the 
smallest set containing the realizations of typical random samples taken from $X$ 
(see Chapter 9 of   \cite{CoverThomas1991}). 
When the differential entropy exists, it takes values in the extended real line $[-\infty, \infty]$, 
whereas the entropy of discrete random variables is always nonnegative. Other incongruities 
have been pointed out in various investigations (see, for instance, \cite{Cufaro} 
and  \cite{Schroeder}). Nevertheless, the use of differential entropy is largely adopted 
in stochastic modeling and applied fields. 
In information theory, large attention is given to the so-called entropy power of a continuous 
random variable $X$, which is a positive quantity expressed in terms of $H(X)$. 
Rather than in stochastic modeling, it is usually adopted 
to compare the differential entropy of a sum of independent random variables with their individual 
differential entropies, and with the entropy of a suitable sum of independent normal random variables 
(see Chapter 16 of    \cite{CoverThomas1991}, and 
\cite{MadimanBarron2007} also for its connection to the Fisher information). 
Hence, the entropy power is useful to analyze stochastic systems governed by unbounded 
random variables that are comparable to Gaussian ones. 
However, in the following sections 
we shall concern mainly with nonnegative random lifetimes. 
\par
Bobkov and Madiman \cite{Bobkov2011} investigated a relevant problem concerning the concentration 
of the information content around the entropy in high dimensions when the pdf of $X$ is log-concave. 
Restricting our attention to the one-dimensional case, hereafter we focus on a relevant quantity related 
to the concentration of $IC(X)$ around $H(X)$, namely the so-called varentropy of $X$, which is defined 
as the variance of the information content of $X$, i.e. 
\begin{eqnarray}\label{eq:varentropy}
V(X) \!\!\!\! & := & \!\!\!\! \mathrm{Var}[IC(X)]= \mathrm{Var}[\log f(X)]
 =\mathbb{E}[(IC(X))^2]-[H(X)]^2 
 \nonumber \\
 & =& \!\!\!\! \int_{-\infty}^\infty{f(x)[\log{f(x)}]^2}\,{\rm d}x- \left[\int_{-\infty}^\infty{f(x)\log{f(x)}}\,{\rm d}x\right]^2.
\end{eqnarray}
The varentropy thus measures the variability in the information content of $X$. The relevance of this 
measure has been pointed out in various investigations, 
especially from Fradelizi {\em et al.}\ \cite{Fradelizi2016}, 
that start from the concept of varentropy of a random variable $X$ and use it to find an optimal 
varentropy bound for log-concave distributions.  Furthermore, 
a sharp uniform bound on varentropy for log-concave distributions is found in the work of 
Madiman \cite{Madiman2014}. An alternative way to calculate a bound for varentropy is 
discussed in  Goodarzi {\em et al.}\ \cite{Goodarzi 2017} where the authors 
use some concepts of reliability theory.  
The generalization from log-concave to convex measures has been studied in the work of Li {\em et al.}\ \cite{Li2016} 
where a bound on the varentropy for convex measures is discussed.  
We recall other works that deal with the bounds of the varentropy in the contest of source coding. 
In particular, Arikan \cite{Arikan2016}, 
analyzing the case of the polar transform, shows that varentropy decreases to zero 
asymptotically as the transform size increases.  
In studies on the lossless source code, 
it is possible to relate varentropy to the dispersion of the source code, as shown  
in the papers by Kontoyiannis and Verd\'u \cite{Kontoyannis2013}, \cite{Kontoyannis2014}, \cite{Verdu2012}. 
Specifically, together with the entropy rate, the varentropy rate serves to tightly approximate the fundamental 
nonasymptotic limits of fixed-to-variable compression for all but very small block lengths. 
\par
We remark that, due to (\ref{eq:shannonentropy}) and (\ref{eq:varentropy}), 
both the entropy and varentropy do not depend on the realization of $X$ but only on its pdf $f$.
\par
In analogy with (\ref{eq:shannonentropy}) and (\ref{eq:varentropy}), the entropy and the varentropy of a discrete 
random variable $X$ taking values in the set $\{x_i ; i\in I\}$ are expressed, respectively, as 
\begin{equation}\label{eq:entropia discreta}
H(X)=\mathbb{E}[IC(X)]=-\sum_{i\in I} \mathbb{P}(X=x_i) \log{\mathbb{P}(X=x_i)} 
\end{equation}
and 
\begin{equation}\label{eq:varentropy discreta}
V(X)= \mathrm{Var}[IC(X)]=\sum_{i\in I} \mathbb{P}(X=x_i) [\log{\mathbb{P}(X=x_i)} ]^2-[H(X)]^2.
\end{equation}
\par
Hereafter, we  analyze an illustrative example related to a three-valued random variable.
\begin{example}\label{es:dueuno}\rm
Let  $X$ be a discrete random variable such that, for a fixed $h>0$, 
\begin{equation}\label{eq:distrXdiscr}
 \mathbb{P}(X=h)=p, \qquad 
 \mathbb{P}(X=0)=1-p-q, \qquad 
 \mathbb{P}(X=-h)=q, 
\end{equation}
with   $0\leq q\leq 1-p\leq 1$. 
Thus, from (\ref{eq:entropia discreta}) and (\ref{eq:varentropy discreta}) we have  
\begin{equation}\label{eq:entropiaesdueuno}
 H(X; p,q)=-p\log{p}-(1-p-q)\log{(1-p-q)}-q\log{q},
\end{equation}
and
\begin{equation}\label{eq:varentropia gauss valori discreti}
 V(X;p,q)=p(\log{p})^2+(1-p-q)[\log{(1-p-q)}]^2+q(\log{q})^2-[H(X;p,q)]^2.
\end{equation}
\par
Figure \ref{fig:varentropiaEs1} shows the varentropy given in (\ref{eq:varentropia gauss valori discreti}) as a 
function of $(p,q)$.
Clearly, it confirms the symmetry property $V(X;p,q)=V(X;q,p)$. We 
can see that the varentropy vanishes in the following 7 cases:  $(p,q,1-p-q)=$ $(0,0,1)$, $(0,1,0)$, $(1,0,0)$, 
$(0.5,0.5,0)$, $(0.5,0,0.5)$, $(0,0.5,0.5)$, $(1/3,1/3,1/3)$. Moreover, the  maximum of $V(X;p,q)$ is 
attained for $(p,q,1-p-q)=$  $(0.06165, 0.06165, 0.8767)$, $(0.8767, 0.06165, 0.06165)$, 
$(0.06165, 0.8767, 0.06165)$. 
\par 
Now consider a system based on the superposition of three  Gaussian signals. Namely, we deal with a 
random variable, say $Y$, whose pdf is a mixture of Gaussian densities with unity variance and mean given 
by $h$, $0$, $-h$ according to the probability law specified in (\ref{eq:distrXdiscr}). 
Hence,  for $x\in\mathbb{R}$, one has 
\begin{equation}\label{eq:pdfY}
 f_Y(x)=(2\pi)^{-1/2}\left[p e^{-(x-h)^2/2}+ (1-p-q)e^{-x^2/2}+ q e^{-(x+h)^2/2}\right].
\end{equation}
Figure \ref{fig:varentropiaEs1bis} shows some instances of the corresponding varentropy as a function of $h$, 
determined numerically by means of (\ref{eq:varentropy}). It can be shown that $V(Y)$ is not monotonic in $h$; 
moreover it reaches large values for the choices of $(p,q)$ that maximize $V(X;p,q)$ and for large values of $h$.  
%
\begin{figure}[t]  
\centering
\includegraphics[scale=0.7]{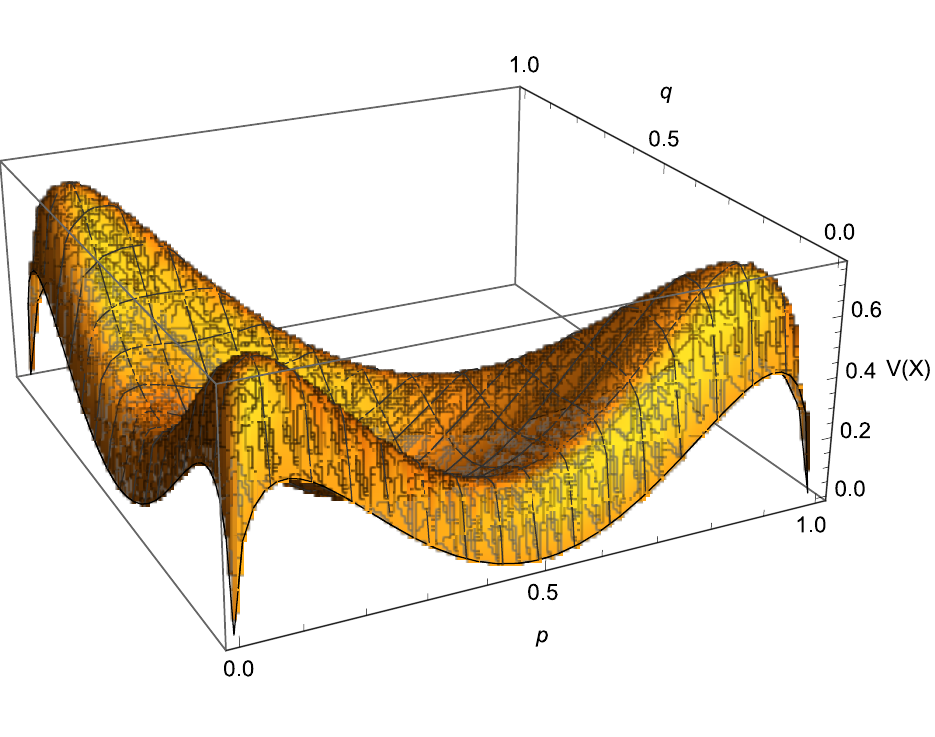} 
$\;$
\includegraphics[scale=0.5]{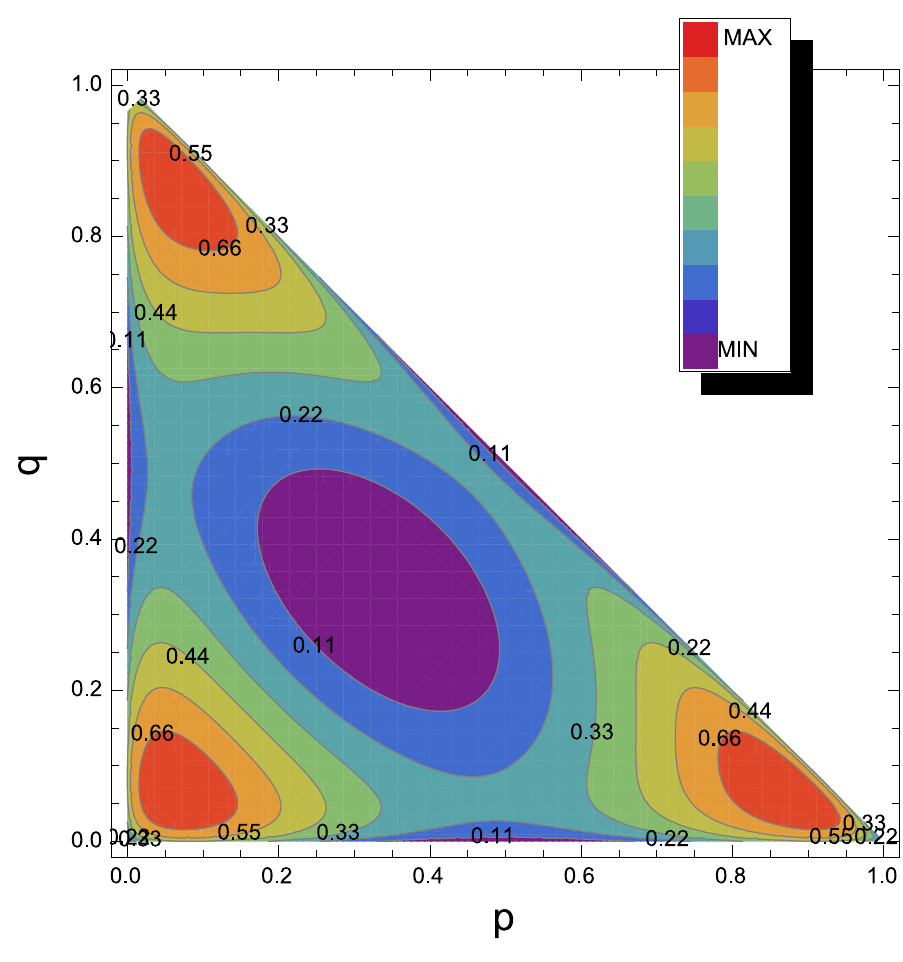} 
\caption{Plots of varentropy \eqref{eq:varentropia gauss valori discreti}; left: 3D plot; right: contourplot.}
\label{fig:varentropiaEs1}
\end{figure}
\begin{figure}[t]  
\centering
\includegraphics[scale=0.8]{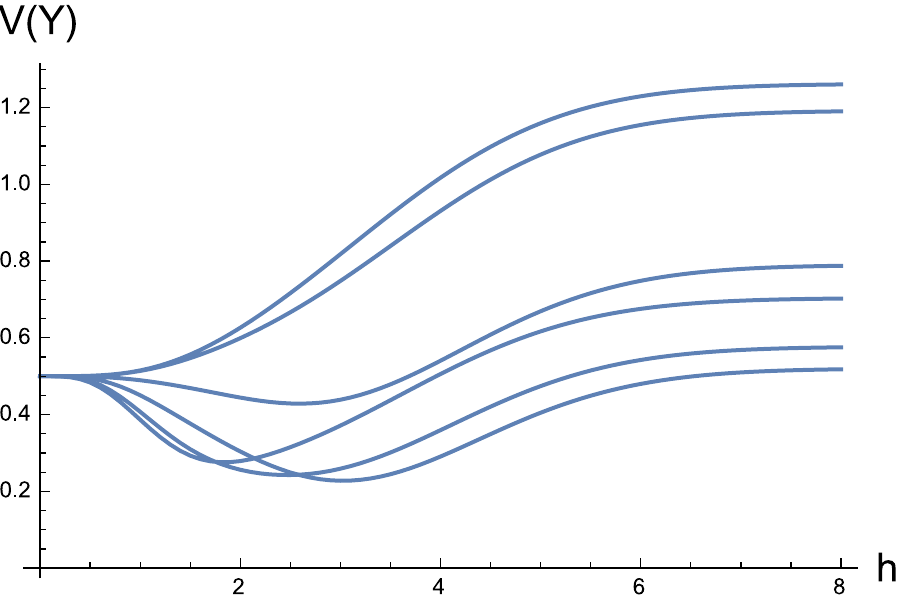} 
\caption{The varentropy corresponding to pdf \eqref{eq:pdfY} for $p=q=0.06165$, $0.1$, $0.2$, $0.45$, $0.4$, $0.3$ 
(from top to bottom for large values of $h$).}
\label{fig:varentropiaEs1bis}
\end{figure}
\end{example}
\par
The relevance of the entropy in information theory and other disciplines is very well known, whereas 
the varentropy has attracted less attention.  Nevertheless, the latter plays a relevant role in the assessment 
of the statistical significance of entropy. Specifically, in the discrete case, the entropy (\ref{eq:entropia discreta}) 
represents the expected number of symbols, in natural base, required to code an event produced by a source 
of information governed by the probability distribution of $X$. In this case, the varentropy (\ref{eq:varentropy discreta})  
measures the variability related to such a coding. In other terms, if two sources of information have the same 
entropy, than the number of digits required in the average to code two sequences produced by such sources is 
the same and is proportional to $H(X)$. However, the number of digits required for a single observed sequence 
in the average is closer to the expected one for the source having the smallest varentropy. Hence, $V(X)$ 
measures how much the entropy is meaningful in the coding of  sequences of symbols generated by $X$. 
\begin{example}\rm 
Let $Y$ be a Bernoulli random variable having distribution $\mathbb{P}(Y=0)=1-\theta$, 
$\mathbb{P}(Y=1)=\theta$, with $0\leq \theta\leq 1$. By means of numerical calculations, it is easy  
to see that for $\theta\approx 0.337009$ one has $H(Y)\approx 0.639032$ and $V(Y)\approx 0.1023$. 
For the distribution considered in the Example \ref{es:dueuno}, if  $p=q=0.1$ from (\ref{eq:entropiaesdueuno}) 
and (\ref{eq:varentropia gauss valori discreti}), we have $H(X)\approx 0.639032$ and $V(X)\approx 0.691852$, 
respectively. Hence, the considered random variables have the same entropy, but the varentropy of $X$ is larger.
This implies that the coding procedure is much more reliable for sequences generated by $Y$. 
\end{example}
%
\subsection{Residual lifetimes}
In order to investigate the role of the varentropy in  reliability theory, we now recall some relevant 
notions in this area. 
Consider a system (such as an item or a living organism) that starts its activity at time 0 
and works regularly up to its failure time. 
Now, we assume that $X$ is a nonnegative absolutely continuous random variable that 
describes the random lifetime of such a system. Hence, 
$H(X)$ is a suitable measure of uncertainty of the failure time. However, the use of $H(X)$ is adequate for a 
brand new system, whereas it is somewhat unrealistic whenever the initial age of the considered system is non-zero. 
In this case, it is appropriate to recall the residual lifetime  
\begin{equation}\label{eq:Xt}
X_t=[X-t|X>t], \qquad t\in D,
\end{equation}
where $D:=
\{t\geq 0: \overline{F}(t)>0\}$. 
Clearly, $X_t$ denotes the system lifetime conditioned to the survival of the system at time $t$. 
The survival function and the pdf of (\ref{eq:Xt}), for any $t\in D$, are given respectively by 
\begin{equation}\label{eq:Ftft}
 \overline F_t(x)=\frac{\overline F(x+t)}{\overline{F}(t)}, 
 \qquad 
 f_t(x)=\frac{f(x+t)}{\overline{F}(t)}, \qquad x>0.
\end{equation}
Hence, recalling (\ref{eq:shannonentropy}), the generalization of the entropy to the residual lifetime 
distributions is given by (see \cite{Ebrahimi1996}, \cite{EbrahimiPellerey1995},  \cite{MuliereP1993})
\begin{equation}\label{eq:entropia}
 H(X_t)= \mathbb{E}[IC(X_t)]
 =-\int_t^\infty{\frac{f(x)}{\overline{F}(t)}\log{\frac{f(x)}{\overline{F}(t)}} \,{\rm d}x}, 
 \qquad  t\in D,
\end{equation}
which is named residual entropy, for short. 
The conventional approach used to characterize the failure distribution of $X$ is either by its (instantaneous) hazard rate function
\begin{equation}\label{eq:hazard rate}
 \lambda(t)=\frac{f(t)}{\overline{F}(t)}
 =\lim_{h\to 0^+} \frac{1}{h} \mathbb P[X\leq t+h|X>t], 
 \qquad t\in D,
\end{equation}
or by its mean residual lifetime function, defined as
\begin{equation}\label{eq:mrl}
 m(t)= \mathbb E(X_t)=\mathbb E[X-t|X>t]=  
 \displaystyle\frac{1}{\overline{F}(t)}\int_t^\infty{\overline{F}(x) \, {\rm d}x}, 
 \qquad t\in D.
\end{equation}
For future needs, we recall also the cumulative hazard rate function of $X$, 
\begin{equation}\label{eq:Lambda}
\Lambda(t)=-\log{\overline{F}(t)}
= \int_0^t \lambda(x)\, {\rm d}x,  \qquad t\in D,
\end{equation}
which plays a relevant role in numerous contexts. Furthermore, we pinpoint the following 
alternative forms of the residual entropy \eqref{eq:entropia}:  
\begin{subequations}\label{eq:entropiaa}
\begin{equation}\label{eq:entropiaalt}
 H(X_t) = -\Lambda(t)-\frac{1} {\overline{F}(t)} \int_t^\infty {f(x)}\log{f(x)} \,{\rm d}x,
\end{equation}
\begin{equation}\label{eq:entropiaalt2}
H(X_t) = 1-\frac{1}{\overline{F}(t)}\int_t^\infty f(x)\log{\lambda(x)} \,{\rm d}x, 
\end{equation}
\end{subequations}
for $t\in D$. Differentiating relation \eqref{eq:entropiaalt}, one has (see, e.g.\ Eq.\ (2.4) of Ebrahimi \cite{Ebrahimi1996})  
\begin{equation}\label{eq:derivata entropia}
	 H'(X_t)=\lambda(t)[H(X_t)-1+\log{\lambda(t)}].
\end{equation}
Moreover, it is known that each of the functions $ \overline{F} $, $\lambda$ and $m$ uniquely determines the other two. 
More specifically, for $t\in D$, we have  
\begin{equation*}
\overline{F}(t) = \exp\left\{- \int_0^t{\lambda(x)\,{\rm d}x}\right\}
= \frac{m(0)}{m(t)} \exp\left\{- \int_0^t{\frac{1}{m(x)}\,{\rm d}x}\right\}, 
\qquad 
\lambda(t)=\frac{m'(t)+1}{m(t)}.
\end{equation*}
%
%
We recall also that Ebrahimi \cite{Ebrahimi1996} showed that  $H(X_t)$ uniquely determines $\overline{F}$ under wide assumptions. 
Useful applications of residual lifetime  distributions in actuarial science 
can be found in Sachlas and Papaioannou \cite{Sachlas2014}.  
%
\section{Residual varentropy}
%
Recalling that the varentropy of a random lifetime $X$ is  defined  
in (\ref{eq:varentropy}), we can now extend the notion of varentropy to the residual lifetime considered in 
(\ref{eq:Xt}). Namely, recalling the second of (\ref{eq:Ftft}), for $t\in D$, we define the varentropy of the 
residual lifetime distribution (residual varentropy, in short) as 
\begin{equation}\label{eq:varentropia}
\begin{split}
V(X_t) &:= {\rm Var}[IC(X_t)]
 = \int_t^\infty{\frac{f(x)}{\overline{F}(t)}\left(\log{\frac{f(x)}{\overline{F}(t)}}\right)^2 {\rm d}x} - [H(X_t)]^2 \\
&=\frac{1}{\overline{F}(t)}\int_t^\infty{f(x)\left[\log{f(x)}\right]^2 {\rm d}x} - \left[\Lambda(t)+H(X_t)\right]^2,
\end{split}
\end{equation}
where $\Lambda(t)$ is given in (\ref{eq:Lambda}), and  $H(X_t)$  is 
provided in (\ref{eq:entropia}) and (\ref{eq:entropiaa}). 
Making use of Eq.\ \eqref{eq:varentropia} we can show, in Table \ref{tab:distribuzioni varentropy costante}, 
some examples in which the residual varentropy is constant.
\begin{table}[h]\centering 
\caption{Selected distributions with constant varentropy. \label{tab:distribuzioni varentropy costante}}
	\begin{tabular}{l c c c}\hline\hline
		\multicolumn{1}{c}{ {Distribution}} &  {Pdf}
		&  {Residual entropy} &  {Residual varentropy}\\ 
		{} & $f(x)$ & $H(X_t)$ & $V(X_t)$ \\
		\hline\\[-5mm]
		Uniform & $\displaystyle\frac{1}{\theta}$ & $\log{(\theta-t)}$ & 0\\[-2mm]
		$D=(0,\theta)$ & & & \\
		Exponential  & $\lambda e^{-\lambda x}$, $\lambda>0$ & $1-\log{\lambda}$ & 1\\[-2mm]
		$D=(0,\infty)$ & & & \\
		Triangular & $2(1-x)$ & $\displaystyle\frac{1}{2}+\log \frac{1-t}{2}$ & $\displaystyle\frac{1}{4}$ \\[-2mm]
		$D=(0,1)$ & & & \\
		\hline\end{tabular}
\end{table}
\par
In the following, we determine the conditions for which the residual varentropy is costant. 
To this aim, we first obtain an expression of its derivative.
\begin{proposition}\label{prop:derivVt}
For all $t\in D$, the derivative of the residual varentropy is 
\begin{equation}
 V'(X_t) =  \lambda(t) \left\{V(X_t) - [H(X_t)+\log \lambda(t)]^2\right\}.
 \label{eq:derivVt}
\end{equation}
\end{proposition}
\begin{proof}
By differentiating both sides of Eq.\ \eqref{eq:varentropia}, and recalling \eqref{eq:hazard rate}, we have 
\begin{eqnarray}
 V'(X_t) \!\!\!\! &=& \!\!\!\! \lambda(t)\left\{ \frac{1}{\overline{F}(t)} \int_t^{\infty} f(x)[\log{f(x)}]^2 \,{\rm d}x -[\log{f(t)}]^2\right\}
 \nonumber \\
 && \!\!\!\! -2[\Lambda(t)+H(X_t)][\lambda(t)+H'(X_t)], \qquad t\in D.
 \label{eq:passaggio 1}
\end{eqnarray}
Then, making use of Eqs.\ (\ref{eq:derivata entropia}) and (\ref{eq:varentropia}),  from (\ref{eq:passaggio 1}) we get 
\begin{eqnarray*}
 V'(X_t) \!\!\!\! &=& \!\!\!\! \lambda(t)\left\{ V(X_t) + [\Lambda(t)+H(X_t)]^2-[\log{f(t)}]^2\right.
\nonumber \\
 && \!\!\!\! \left. -2[\Lambda(t)+H(X_t)][H(X_t)+\log \lambda(t)] \right\}, \qquad   t\in D.
\end{eqnarray*}
Hence, due to (\ref{eq:Lambda}), after some calculations, we obtain Eq.\ (\ref{eq:derivVt}). 
\end{proof}
\par
As a consequence of Proposition \ref{prop:derivVt}, we can now provide some useful results 
involving  the residual varentropy, the residual entropy, the hazard rate, and the varentropy of a lifetime $X$. 
\begin{theorem}\label{th:teorema varentropia costante}
Let  $X$ have a pdf such that $f(t)>0$ for all  
$t\in (0,r)$, with $r\in (0,\infty]$. 
\\
(i) \ If the residual varentropy $V(X_t)$ is constant, say 
\begin{equation}\label{eq: varentropia costante}
 V(X_t)=v \geq 0, \qquad \forall 
 t\in [0,r),
\end{equation}
then the following relation holds: 
\begin{equation}\label{eq:condizione varentropia costante}
 |H(X_t)+\log{\lambda(t)} |=\sqrt{v}, \qquad \forall 
 t\in (0,r).
\end{equation}
(ii) \ Let $c\in \mathbb R$; 
if 
\begin{equation}\label{eq:condizione varentropia costante new}
 H(X_t)+\log{\lambda(t)} =c, \qquad \forall t\in  (0,r),
\end{equation}
then 
\begin{equation}\label{eq: varentropia costante new}
 V(X_t)=c^2 +\frac{V(X)-c^2}{\overline{F}(t)},\qquad \forall t\in [0, r).
\end{equation}
\end{theorem}
\begin{proof}
Since $f(t)>0$ for all $t\in (0,r)$, 
the assumption (\ref{eq: varentropia costante}) immediately 
gives (\ref{eq:condizione varentropia costante}), due to (\ref{eq:derivVt}). 
Moreover, if condition (\ref{eq:condizione varentropia costante new}) holds, then 
Eq.\ (\ref{eq:derivVt}) becomes 
$$
  V'(X_t) =  \lambda(t) \left\{V(X_t) - c^2\right\}, \qquad t \in  (0,r),
$$
with initial condition $ V(X_t)|_{t=0}=V(X)$. Finally, it is not hard to see that the solution of such problem 
yields Eq.\  (\ref{eq: varentropia costante new}). 
\end{proof}

\par
Let us now recall the notion of {\em generalized hazard (or failure) rate} of $X$  expressed by 
(see Schweizer and Szech \cite{Schweizer2015}) 
\begin{equation}\label{eq:genfailrate}
	\lambda_{\alpha}(t)=\frac{f(t)}{[\overline{F}(t)]^{1+\alpha}}, 
	\qquad t\in D,
\end{equation}
for $\alpha\in\mathbb{R}$. Clearly, recalling (\ref{eq:hazard rate}), one has $\lambda_{0}(t)=\lambda(t)$ for all $t$. 
Other parameterizations of $\lambda_{\alpha}(t)$ have been treated in Bieniek and Szpak  \cite{BieniekSzpak}
as  a special case of the generalized failure rate defined by Barlow and van Zwet \cite{BarlowvanZwet}. 
Further forms of generalized hazard rates have been considered in the past. 
For instance, Lariviere and Porteus \cite{LaPo2001}, and Maoui {\em et al}.\ \cite{Maoui} considered $t\,\lambda(t)$ 
as generalized hazard rate. Moreover, a different version has been treated in Li and Tewari \cite{LiTe2018}. 
\par
We are now able to provide necessary and sufficient conditions  in terms of the residual entropy 
(cf.\ point (ii) of Theorem \ref{th:teorema varentropia costante}), 
such that the generalized  hazard rate of $X$ is constant.  
Recall that $H(X)$ denotes the entropy given in (\ref{eq:shannonentropy}). 
\begin{theorem}\label{th.const}
Let $X$ possess a pdf such that $f(t)>0$ for all $t\in (0,r)$, with $r\in (0,\infty]$. The generalized hazard rate of $X$ is constant, such that  
\begin{equation}\label{eq:tesi corollario varentropia costante}
	\lambda_{c-1}(t)=  e^{c-H(X)}, \qquad  t\in [0,r),
\end{equation}
if and only if Eq.\ \eqref{eq:condizione varentropia costante new} is fulfilled for a given $c\in\mathbb{R}$. 
\end{theorem}
\begin{proof}
Assume that the Eq.\ \eqref{eq:tesi corollario varentropia costante} is fulfilled. Making use of (\ref{eq:hazard rate}) and 
(\ref{eq:entropiaalt}), we have 
\begin{equation}\label{eq:primo membro condizione entropia}
 H(X_t)+\log{\lambda(t)} = \log{f(t)}+\frac{1}{\overline{F}(t)}\left\{H(X)+\int_0^t f(x)\log{f(x)}\,\mathrm{d}x\right\}.
\end{equation}
From the assumption \eqref{eq:tesi corollario varentropia costante} it is not hard to see that 
$$
 \int_0^t f(x)\log{f(x)}\,\mathrm{d}x = -F(t)\, H(X)-c \,\overline{F}(t)\log \overline{F}(t).
$$
Hence, due to Eqs.\ (\ref{eq:tesi corollario varentropia costante}) 
and (\ref{eq:primo membro condizione entropia}), we have 
$$
 H(X_t)+\log{\lambda(t)} =  H(X)+\log \frac{f(t)}{[\overline{F}(t)]^c}=c,
$$
so that  \eqref{eq:condizione varentropia costante new} holds. 
Now, let us prove that  (\ref{eq:condizione varentropia costante new}) implies the validity of 
Eq.\ \eqref{eq:tesi corollario varentropia costante}. In fact, rearranging Eq.\ (\ref{eq:derivata entropia}), we have 
$$
 H(X_t)+\log \lambda(t)=\frac{H'(X_t)}{\lambda(t)}+1,
$$
so that, due to Eq.\ (\ref{eq:condizione varentropia costante new}), one has 
$$
 H'(X_t)=(c-1)  \lambda(t), \qquad t\in (0,r). 
$$
By integration over $[0,t]$, and recalling (\ref{eq:Lambda}), one obtains 
$$
 H(X_t)-H(X)=(c-1) \Lambda(t), \qquad t\in [0,r).
$$ 
Comparing the latter identity with Eq.\ (\ref{eq:condizione varentropia costante new}) 
and in virtue of (\ref{eq:Lambda}), after some algebraic calculations, we get 
$$
 \log \frac{f(t)}{[\overline{F}(t)]^{c}}=c-H(X),
$$
which gives immediately relation 
\eqref{eq:tesi corollario varentropia costante} by virtue of (\ref{eq:genfailrate}). 
\end{proof}
\begin{remark}\rm 
(i) 
It is worth pointing out that, due to Theorem 3.1 of Asadi and  Ebrahimi \cite{AsadiEbrahimi2000}, 
the condition expressed in Eq.\ (\ref{eq:condizione varentropia costante new}) is 
fulfilled if and only if $X$ has a generalized Pareto distribution, with survival function  
\begin{equation}
 \overline F(t)=\left(\frac{b}{a t + b}\right)^{\frac{1}{a}+1}, 
 \qquad t\geq 0,
 \label{eq:genParetosurvf}
\end{equation}
for $a>-1$ and $b>0$. 
The generalized Pareto distribution is a flexible statistical model which is employed in several research areas, such as statistical physics, econophysics and social sciences, since its distribution possesses a tail of general form.
Specifically, it includes the exponential distribution ($a\to 0$), the Pareto distribution ($a> 0$, with heavy tail), 
and the power distribution ($-1<a< 0$, with bounded support). An intuitive reason leading to the above 
result is due to the property that the generalized Pareto distribution is the only family of distributions 
whose mean residual function (\ref{eq:mrl}) is linear (see Hall and Wellner \cite{HallWellner1981}). 
Indeed, for the survival function (\ref{eq:genParetosurvf}) we have $m(t)= a t + b$, with hazard rate function 
$\lambda(t) = \frac{1 + a}{a t + b}$. For a recent characterization of this distribution in the context of shape 
functionals, see Arriaza {\em et al.}\ \cite{Arriaza2019}.  
\\
(ii) 
A special case arises from (\ref{eq:genParetosurvf}) in the limit as $a\to\infty$ and  $b\to\infty$, 
with $\frac{a}{b}\to\lambda>0$, by which the pdf and the survival function of $X$ are given, 
respectively, by 
$$
 f(t)=\frac{\lambda}{(1+\lambda t)^2}, \qquad 
 \overline{F}(t)=\frac{1}{1+\lambda t}, \qquad  t\in[0,\infty).
$$
In this case, $X$ has a modified Pareto distribution that describes the first arrival time in a Geometric 
counting process with parameter $\lambda>0$ (cf.\ Section 2.2 of \cite{DiCrPell}, for instance). 
From Eq.\ \eqref{eq:genfailrate}, it immediately follows that the generalized hazard rate of $X$ 
is a constant for $\alpha=1$, i.e.\ $\lambda_1(t)\equiv \lambda$. As a consequence, 
Eq.\ \eqref{eq:tesi corollario varentropia costante} is fulfilled for $c=2$ and $H(X)=2-\log \lambda$. 
From Theorems \ref{th:teorema varentropia costante} and \ref{th.const},  we thus obtain the 
(increasing) residual entropy, 
$$
 H(X_t)=2-\log \frac{\lambda}{1+\lambda t}, 
 \qquad t\geq 0,
$$
and the corresponding constant residual varentropy, $V(X_t)=4$. 
It is worth pointing out that in this special case, the mean residual lifetime is infinite. Hence, 
for such a stochastic model the residual entropy and the residual varentropy provide 
useful information even if the mean residual lifetime is not finite. 
\end{remark}
\par
The following example is concerning a family of distributions for which the residual varentropy exhibits 
different behaviors.
\begin{example}\label{example:Weibull}\rm 
Let $X_{\lambda,k}$ have  Weibull distribution, with pdf 
\begin{equation}\label{eq:weibull pdf}
f_{\lambda,k}(x)=\frac{k}{\lambda} \left(\frac{x}{\lambda}\right)^{k-1} e^{-({x}/{\lambda})^k}, 
\qquad x>0,
\end{equation}
where $k>0$ is the shape parameter and $\lambda>0$ is the scale parameter. 
Recall that this family of distributions includes special cases of interest, such as the 
exponential distribution (for $k=1$) and the Rayleigh distribution (for $k=2$). 
A characterization of the Weibull distribution in terms of a Gini-type index 
of interest in reliability theory is provided in Theorem 1 of \cite{Parsa2018}. 
The expression of the residual varentropy is omitted being quite cumbersome.
The behavior of the pdf (\ref{eq:weibull pdf}) and of the corresponding residual varentropy 
is visualized in Fig.\ \ref{fig:weibull} for some choices of the shape parameter.
It can be seen that the residual varentropy is decreasing, constant, increasing, non monotonic  
for $k=0.5$, $1$, $1.5$, $3.5$ respectively.

\begin{figure}[t]  
\centering
\includegraphics[scale=0.35]{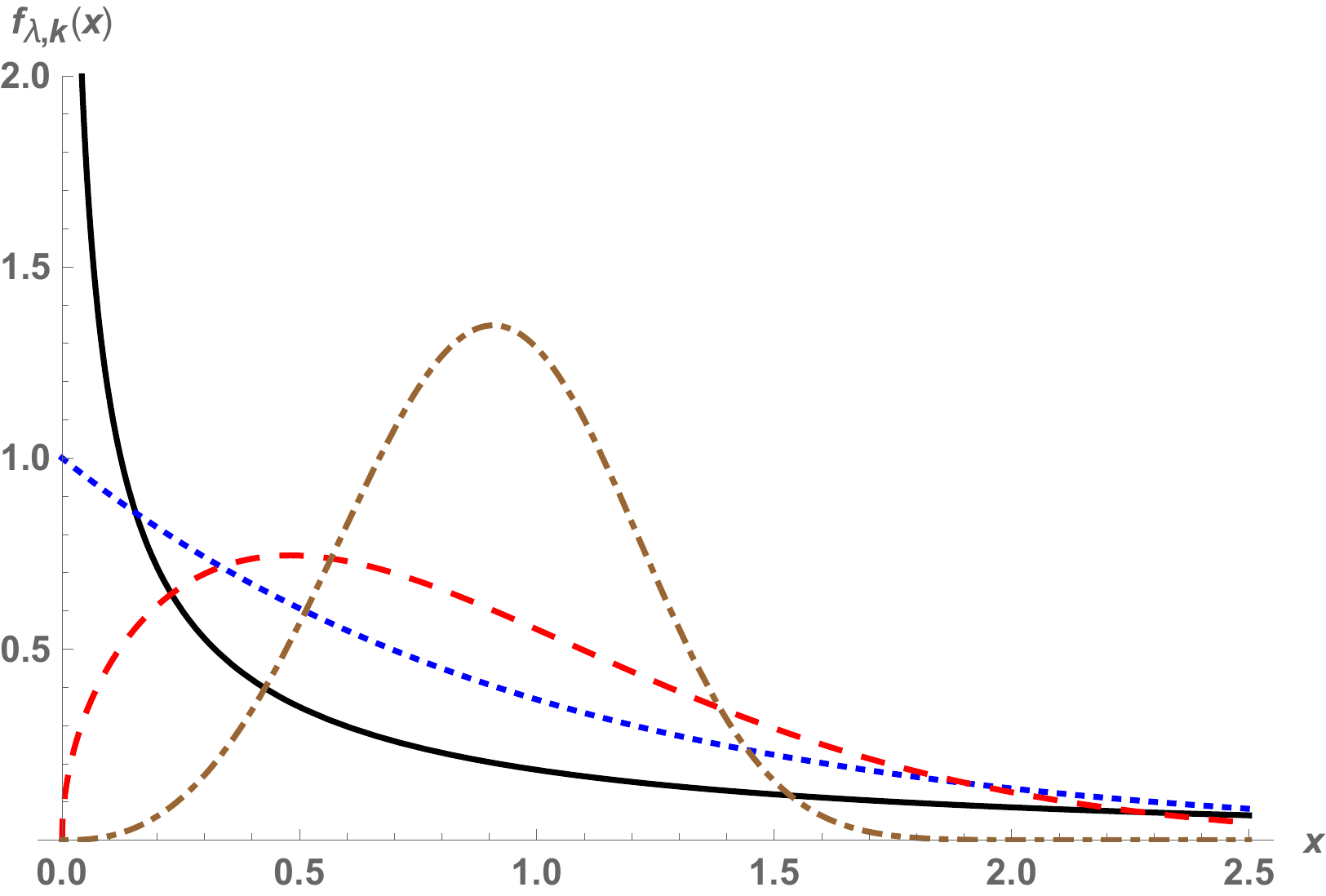}
$\,$
\includegraphics[scale=0.35]{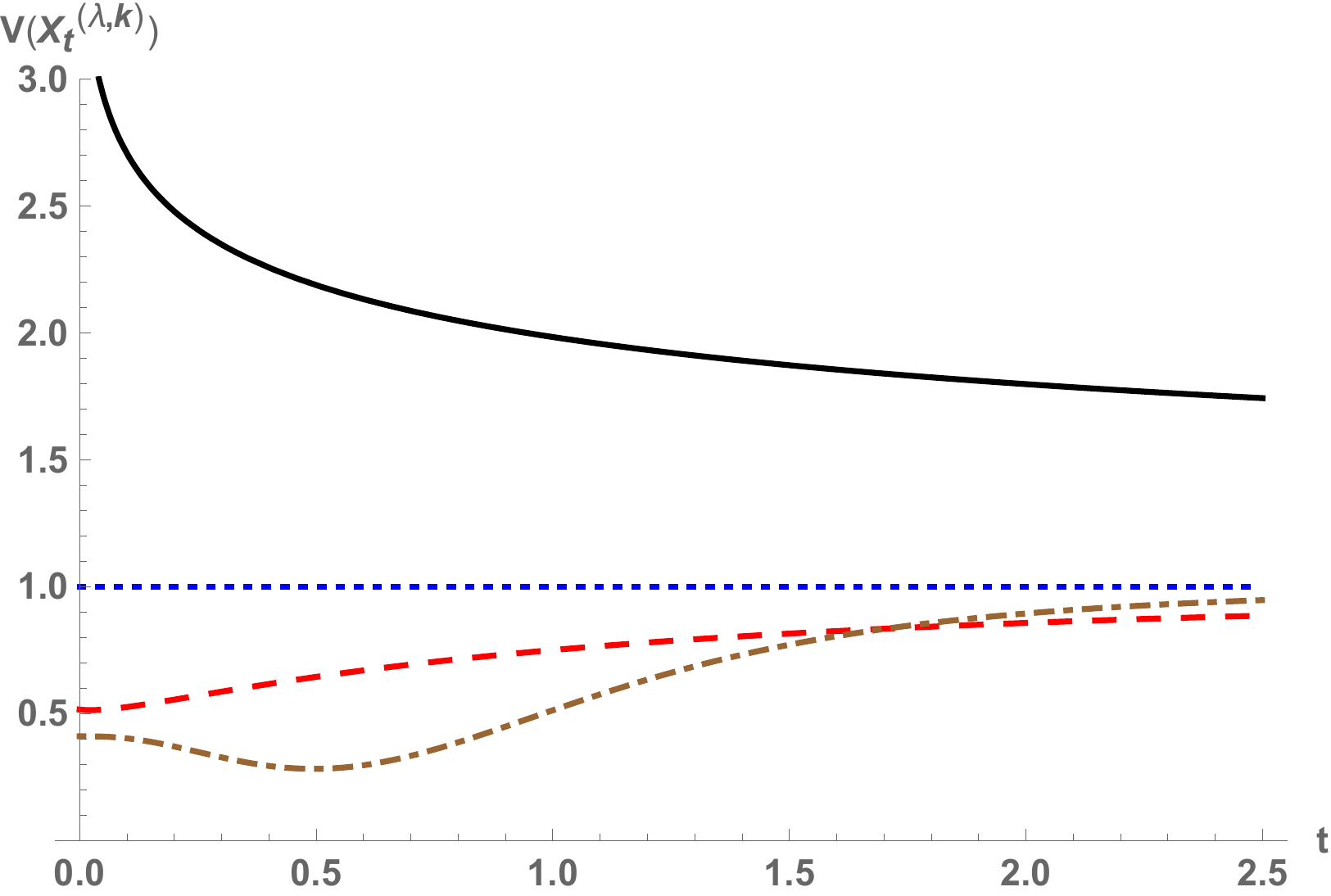}
$\,$
\includegraphics[scale=0.4]{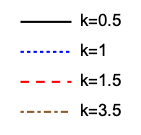}
%
\caption{(left) Weibull pdf, given in \eqref{eq:weibull pdf}, and (right)
	residual varentropy for $\lambda=1$ and various choices of $k$ 
	(as indicated in the label).}
\label{fig:weibull}
\end{figure}

\end{example}
\par
Let us now analyze the effect of linear transformations to the residual varentropy. 
We recall that if 
\begin{equation}\label{eq:definizione Yt}
	Y=a X + b, \qquad a>0, \quad b\geq 0,
\end{equation}
then the residual entropy of $X$ and $Y$ are related by (see Eq.\ (2.6) of Ebrahimi and Pellerey \cite{EbrahimiPellerey1995}) 
\begin{equation}\label{eq:relresentr}
 H(Y_t)=H\left(X_{\frac{t-b}{a}}\right)+\log a, \qquad \forall \  t. 
\end{equation}
\begin{proposition}\label{prop: trasformazioni lineari}
Let $X$ and $Y$ be related by (\ref{eq:definizione Yt}). Hence, for their residual varentropies, we have:
\begin{equation}\label{eq:tesi proposizione trasformazione lineare}
	V(Y_t)=V\left(X_{\frac{t-b}{a}}\right) \qquad \forall \ t .
\end{equation}
\end{proposition}
\begin{proof}
Clearly, from (\ref{eq:definizione Yt}) we have that the cdfs and the pdfs of $Y$ and $X$ are related by 
$F_Y(x)=F_X\left(\frac{x-b}{a}\right)$ and $f_Y(x)=\frac{1}{a} \ f_X\left(\frac{x-b}{a}\right)$. 
Hence, recalling \eqref{eq:varentropia} and (\ref{eq:relresentr}), it is not hard to see that 
$$
 V(Y_t)=\int_{\frac{t-b}{a}}^\infty{ \frac{f_X(x)}{\overline{F}_X(\frac{t-b}{a})} 
	\left[\log{\frac{f_X(x)}{\overline{F}_X(\frac{t-b}{a})}}-\log{a}\right]^2\mathrm{d}x}
	-\left[H\left(X_{\frac{t-b}{a}}\right)+\log{a}\right]^2. 
$$
The thesis (\ref{eq:tesi proposizione trasformazione lineare}) thus follows after some calculations. 
\end{proof}
%
\subsection{Bounds}
We conclude this section by discussing some bounds to the residual varentropy. 
\par
First, we provide a lower bound for $V(X_t)$. It will be expressed in terms of the 
``variance residual life function'', defined as the variance of (\ref{eq:Xt}), that is, 
\begin{equation}\label{eq:sigma2t}
\sigma^2(t)={\rm Var}(X_t)={\rm Var}[X-t|X>t]
=\frac{2}{\overline{F}(t)}\int_t^{\infty}{\rm d}x\int_x^{\infty} \overline{F}(y)\,{\rm d}y-[m(t)]^2,
\end{equation}
with $m(t)$ defined in (\ref{eq:mrl}). 
For instance, see Gupta \cite{Gupta2006} for characterization results  and properties of 
$\sigma^2(t)$.  
\begin{theorem}\label{Th:lowerbound}
Let $X_t$ be a residual lifetime as defined in (\ref{eq:Xt}), 
and assume that the corresponding mean residual lifetime $m(t)$ and variance residual 
lifetime $\sigma^2(t)$ are finite (cf.\ (\ref{eq:mrl}) and (\ref{eq:sigma2t}), respectively). Then, for all $t\in D$, 
\begin{equation}\label{eq:boundVXt}
V(X_t)\geq \sigma^2(t)\, (\mathbb E[w_t'(X_t)])^2
\end{equation}
where the function $w_t(x)$ is defined by 
$$
 \sigma^2(t)\,w_t(x)\, f_{t}(x)=
 \int_{0}^x 
 {[m(t)-z]\, f_{t} (z)\,{\rm d}z},
 \qquad 
 x>0,
$$
with $f_{t}(x)$ given in the second of (\ref{eq:Ftft}).
\end{theorem}
\begin{proof}
We recall that if $X$ is an absolutely continuous random variable with pdf $f(x)$, mean $\mu$ 
and variance $\sigma^2$, then (cf.\ Cacoullos and Papathanasiou \cite{Cacoullos1989}) 
\begin{equation}\label{eq:boundG}
 {\rm Var}[g(X)] \geq \sigma^2 (\mathbb{E}[w(X) g'(X)])^2,
\end{equation}
where $w(x)$ is defined by  $\sigma^2 w(x) f(x)= 
\int_{0}^x (\mu-z)f(z)\,{\rm d}z$. 
Hence, by taking $X_t$ as 
reference, with $g(x) =- \log f(x)$ and integrating by parts, similarly as Eq.\ (3.9) of Goodarzi {\em et al.}\ \cite{Goodarzi 2017}, 
we obtain (\ref{eq:boundVXt}). 
\end{proof}
\par
Note that the equality in (\ref{eq:boundG}) holds if and only if $X$ is exponentially distributed. 
\par
Hereafter, we determine suitable upper bounds to the residual varentropy, 
thus providing conditions on its finiteness. First, we recall that $X$ is said to be ILR 
(increasing in likelihood ratio) if its pdf $f(x)$ is such that $\log f(x)$ is  a concave 
function on $(0,\infty)$; equivalently, we say that $X$ has a log-concave pdf. 
\begin{theorem}\label{Th:upperbound}
Given a random lifetime $X$ with log-concave pdf $f(x)$, then
$$
 V(X_t)\leq 1, \qquad \hbox{for all $t\in D$}.
$$ 
\end{theorem}
\begin{proof}
We note that if $f(x)$ is log-concave, then also $f_t(x)$ is log-concave due to (\ref{eq:Ftft}). Hence, 
the proof  is a direct consequence of Theorem 2.3 of Fradelizi {\em et al.}\ \cite{Fradelizi2016}, which 
states that the varentropy of a random lifetime with log-concave pdf is not greater than 1. 
\end{proof}
\par
The following bound is expressed in terms of the weighted residual entropy of $X$, which is a 
weighted version of the residual entropy (\ref{eq:entropia}) and is given by 
(see Di Crescenzo and Longobardi \cite{DiCrLong} for details)
\begin{eqnarray}
 H^w(X_t) \!\!\!\! &=& \!\!\!\! -\int_t^\infty x \,{\frac{f(x)}{\overline{F}(t)}\log{\frac{f(x)}{\overline{F}(t)}} \,{\rm d}x} 
 \nonumber \\
 &=& \!\!\!\! -\frac{1}{\overline{F}(t)} \int_t^\infty  x\,f(x) \log{f(x)} {\rm d}x 
 -\frac{\Lambda(t)}{\overline{F}(t)}\int_t^\infty  x\,f(x) {\rm d}x,  
 \qquad  t\in D.
 \label{eq:HwXt}
\end{eqnarray}
Furthermore, it is based on the so-called vitality function of $X$, i.e.
\begin{equation}
 \delta(t):= \mathbb{E}[X|X>t] = m(t)+t, 
 \qquad  t\in D.
 \label{eq:defdelta}
\end{equation}
Namely, since $X$ denotes the random lifetime of a system, $\delta(t)$ 
can be interpreted as the average life span of a system whose age exceeds $t$. 
\begin{theorem}\label{Th:lowerbound2}
If $X$ is a random lifetime such that its pdf satisfies 
\begin{equation}
 e^{-\alpha x-\beta} \leq f(x)\leq 1 \qquad \forall x\geq 0,
 \label{eq:assfx}
\end{equation}
with $\alpha >0$  and $\beta\geq 0$, then for all $t\geq 0$
\begin{equation}
 V(X_t)\leq \alpha [\Lambda(t) \delta(t)+H^w(X_t)]
 +\beta[ \Lambda(t)+H(X_t)]
 -[ \Lambda(t)+H(X_t)]^2.
 \label{eq:tesi}
\end{equation}
\end{theorem}
\begin{proof}
From Eq.\ (\ref{eq:varentropia}), due to (\ref{eq:assfx}) one has 
\begin{equation}
V(X_t) 
 \leq - \frac{1}{\overline{F}(t)}\int_t^\infty (\alpha x+\beta)
 f(x) \log{f(x)} {\rm d}x - \left[\Lambda(t)+H(X_t)\right]^2,
 \qquad t\geq 0.
 \label{eq:b1}
\end{equation}
We note that Eqs.\ (\ref{eq:mrl}) and (\ref{eq:defdelta}) give
\begin{equation*}
 \int_t^\infty  x\,f(x) {\rm d}x 
 = \overline{F}(t)  \delta(t),
 \qquad t\geq 0.
\end{equation*}
Hence, recalling (\ref{eq:Lambda}) and (\ref{eq:defdelta}), Eq.\ (\ref{eq:HwXt}) implies:
\begin{equation}
 \int_t^\infty  x\,f(x) \log{f(x)} {\rm d}x 
 = -\overline{F}(t) [\Lambda(t)\delta(t)+H^w(X_t)],
 \qquad t\geq 0.
 \label{eq:b2}
\end{equation}
Moreover, from (\ref{eq:entropiaalt}), we have 
\begin{equation}
 \int_t^\infty  f(x) \log{f(x)} {\rm d}x 
 = -\overline{F}(t) [\Lambda(t)+H(X_t)],
 \qquad t\geq 0.
 \label{eq:b3}
\end{equation}
Finally, substituting (\ref{eq:b2}) and (\ref{eq:b3}) in  (\ref{eq:b1}), we immediately obtain 
the inequality (\ref{eq:tesi}). 
\end{proof}
%
\section{Some applications}
In this section, we consider some applications of the residual varentropy.  
We first deal with the proportional hazard rates model, which in turn 
can be employed to the reliability analysis of series systems. A further case of interest 
is concerning the first-passage-time problem of an Ornstein-Uhlenbeck jump-diffusion process 
which arises as a limit of the continuous-time Ehrenfest model. 
\subsection{Proportional hazards model}
Consider a family of absolutely continuous nonnegative 
random variables $\{X^{(a)}; a>0 \}$, where the survival function 
and the pdf of $X^{(a)}$ are expressed, respectively, as 
\begin{equation}\label{eq:prohrm}
 \overline{F}^{(a)}(t)=\mathbb{P}[X^{(a)}>t]=[\overline{F}(t)]^a, 
 \quad 
 f^{(a)}(t)=a[\overline{F}(t)]^{a-1} f(t), 
 \qquad   
 t>0,
\end{equation}
with $\overline{F}(t)$ a suitable baseline survival function and  $f(t)=-\frac{\mathrm{d}}{\mathrm{d} t} \overline{F}(t)$ 
the associated pdf. This model is known as the proportional hazards model, 
see Cox \cite{Cox1959}, since the hazard rate function of $X^{(a)}$ is proportional to the hazard rate 
corresponding to the baseline survival function. For instance, see Parsa {\em et al.}\ \cite{Parsa2018} 
for a recent characterization of the proportional hazards model in terms of the Gini-type index. 
\par
Let us now address the problem of evaluating the residual varentropy for the model (\ref{eq:prohrm}) 
when $X^{(a)}$ is a random lifetime. First, noting that the cumulative hazard rate function is given by 
\begin{equation}\label{eq:cumulativa hazard rate sistema serie}
 \Lambda^{(a)}(t)=-\log{\overline{F}^{(a)}(t)}=a \, \Lambda(t), \qquad t>0,
\end{equation}
from (\ref{eq:entropiaalt}), it is not hard to see that the residual entropy of $X^{(a)}$ is expressed as  
\begin{eqnarray}\label{eq:entropia sistema serie}
 H(X_t^{(a)}) \!\!\!\! &=& \!\!\!\! -\Lambda^{(a)}(t)-\frac{1}{[\overline{F}(t)]^a}\int_t^\infty{f^{(a)}(x)\log{f^{(a)}(x)}}\,{\rm d}x  
 \nonumber \\
 &=& \!\!\!\! -a \,\Lambda(t)-\frac{1}{[\overline{F}(t)]^a}\int_0^{[\overline F(t)]^a}  \ell(y;a)\,{\rm d}y, \qquad t>0,
\end{eqnarray}
with $y=[\overline{F}(x)]^a$, and where 
\begin{equation}\label{eq:defell}
 \ell(y;a) := \log{\left\{a \, y^{1-1/a} \, f[\overline{F}^{-1}(y^{1/a})]\right\}},
 \qquad 0<y<1.
\end{equation}
Hence, recalling (\ref{eq:varentropia}), from (\ref{eq:cumulativa hazard rate sistema serie}) and 
(\ref{eq:entropia sistema serie}) after some calculations, 
we obtain the residual varentropy of $X^{(a)}$, for $t>0$: 
\begin{eqnarray}\label{eq:varentropy sistema serie}
 V(X_t^{(a)}) \!\!\!\!
 & =& \!\!\!\!
 \frac{\int_t^\infty{f^{(a)}(x)[\log{f^{(a)}(x)}}]^2\,{\rm d}x}{[\overline{F}(t)]^a}   
    -\left[\frac{\int_t^\infty{f^{(a)}(x)\log{f^{(a)}(x)}}\,{\rm d}x}{[\overline{F}(t)]^a} \right]^2
 \nonumber 
 \\
  & =& \!\!\!\! \frac{1} {[\overline{F}(t)]^a} \int_0^{[\overline{F}(t)]^a}   [\ell(y;a)]^2  \mathrm{d}y  
    -\left\{\frac{1}{ [\overline{F}(t)]^a}\int_0^{ [\overline{F}(t)]^a} \ell(y;a) \, \mathrm{d}y\right\}^2.\quad 
 \end{eqnarray}
Making  use of Eqs.\ \eqref{eq:hazard rate} and  \eqref{eq:Lambda}, one has $f(x)=\lambda(x) e^{-\Lambda(x)}$, 
so that the function introduced in (\ref{eq:defell}) can be rewritten also as follows: 
$$
 \ell(y;a)=\log{\Big\{a y \lambda \Big(\Lambda^{-1}\Big(-\frac{1}{a}\log{y}\Big)\Big)\Big\}}. 
$$
\par
An application can be immediately given to series systems. 
\begin{example}\label{ex:series} \rm 
Consider a system composed of $n$ units in series and characterized by i.i.d.\ random lifetimes 
$X_1, \ldots, X_n$. Let the survival function of each unit  be denoted with $\overline{F}(t)=\mathbb{P}(X_i>t)$. 
Since the system lifetime is given by  $X^{(n)}=\min\{X_1, \ldots, X_n\}$, 
the model  of  series system satisfies the  proportional
hazards model specified in (\ref{eq:prohrm}), for $a=n\in \mathbb{N}$. 
\par
For an illustrative example, we assume that the random lifetimes $X_i$ have generalized exponential 
distribution with survival function $\overline F(t)=1 - (1 - e^{-\lambda t})^b$, $t\geq 0$, for $b>0$. 
(We recall that this distribution plays a role in the construction of probabilistic models for damped random 
motions with finite velocities \cite{DiCrMart}). From (\ref{eq:defell}), thus  we  have 
$$
   \ell(y;a)  = \log \left\{a b \lambda  y^{1-\frac{1}{a}} (1-y)^{1-\frac{1}{b}}
   \big[1-(1-y)^{\frac{1}{b}}\big]
   \right\}, \qquad 0<y<1.
$$
From Eq.\ (\ref{eq:varentropy sistema serie}), we come to the  residual varentropy of the 
system lifetime $X^{(n)}$. The expression of $V(X_t^{(a)})$ cannot be obtained in closed form, 
but it can be evaluated via numerical computations. Figure \ref{fig:PRH} shows some plots of 
the  residual varentropy for some choices of $a=n$. It is clear that the varentropy increases when
the number of units grows, and generally when  $t$ becomes larger. 
%
\begin{figure}[t]  
\centering
\includegraphics[scale=0.31]{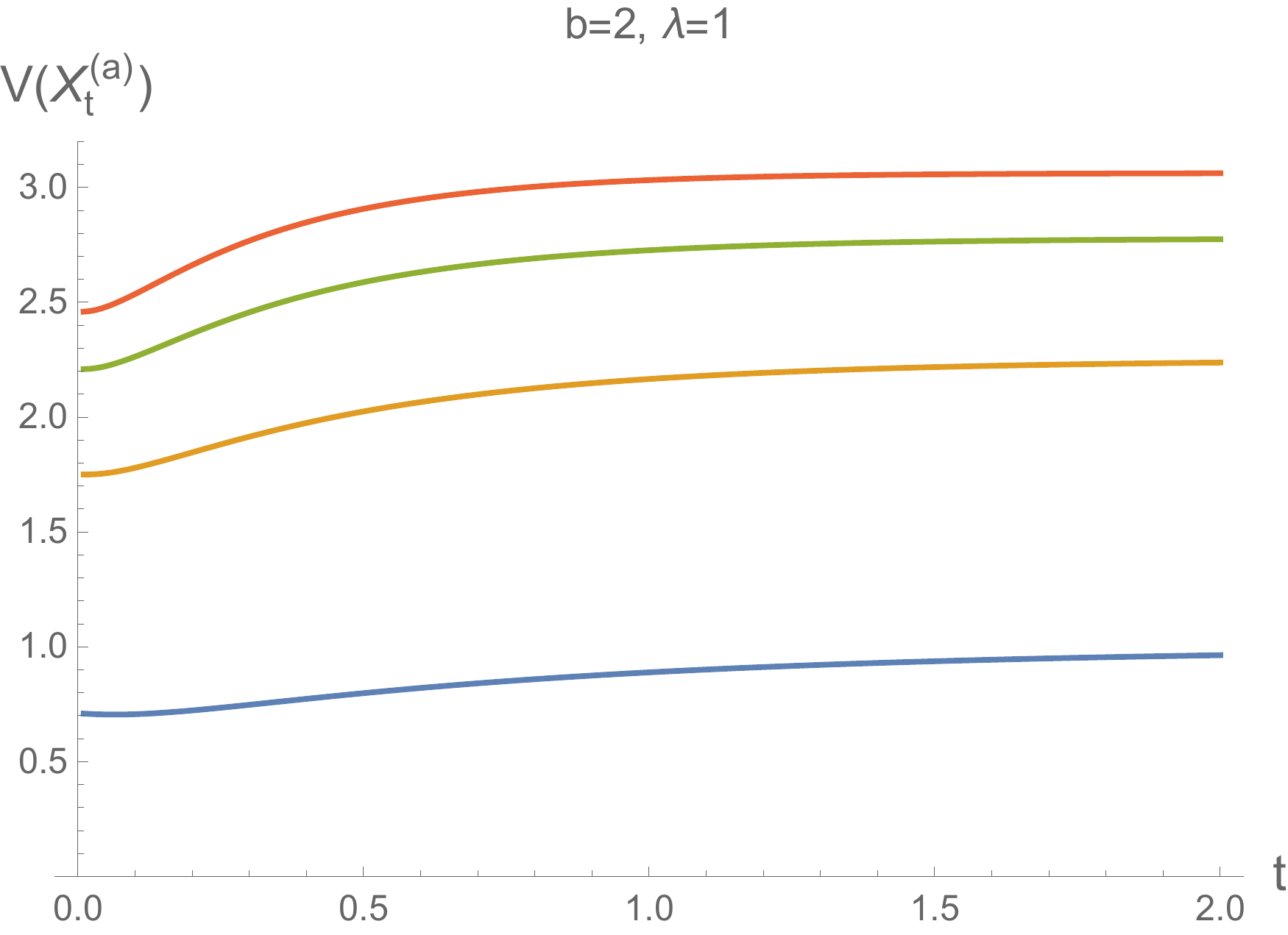}
$\,$
\includegraphics[scale=0.31]{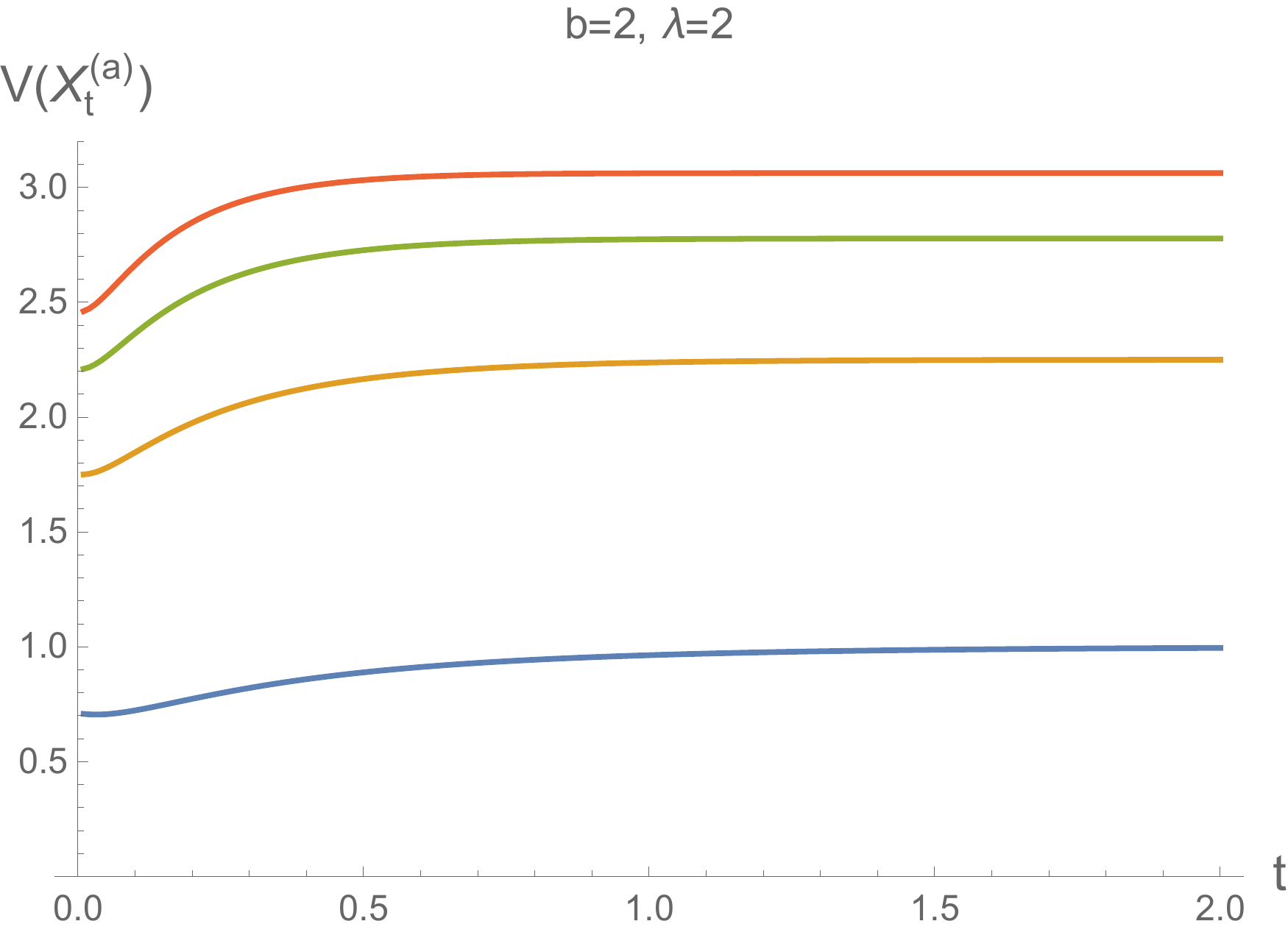}
%

\vspace{0.5cm}

\includegraphics[scale=0.31]{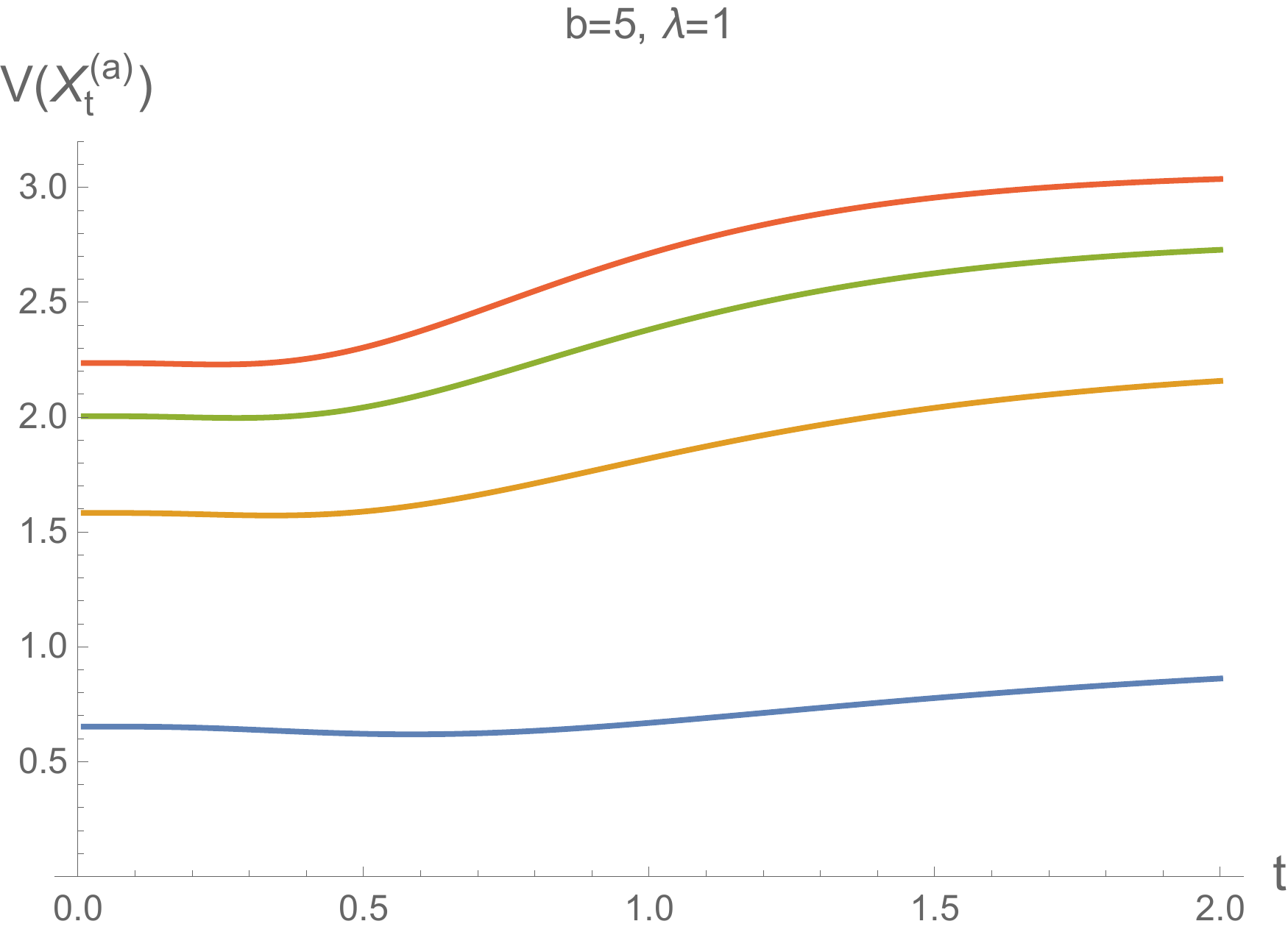}
$\,$
\includegraphics[scale=0.31]{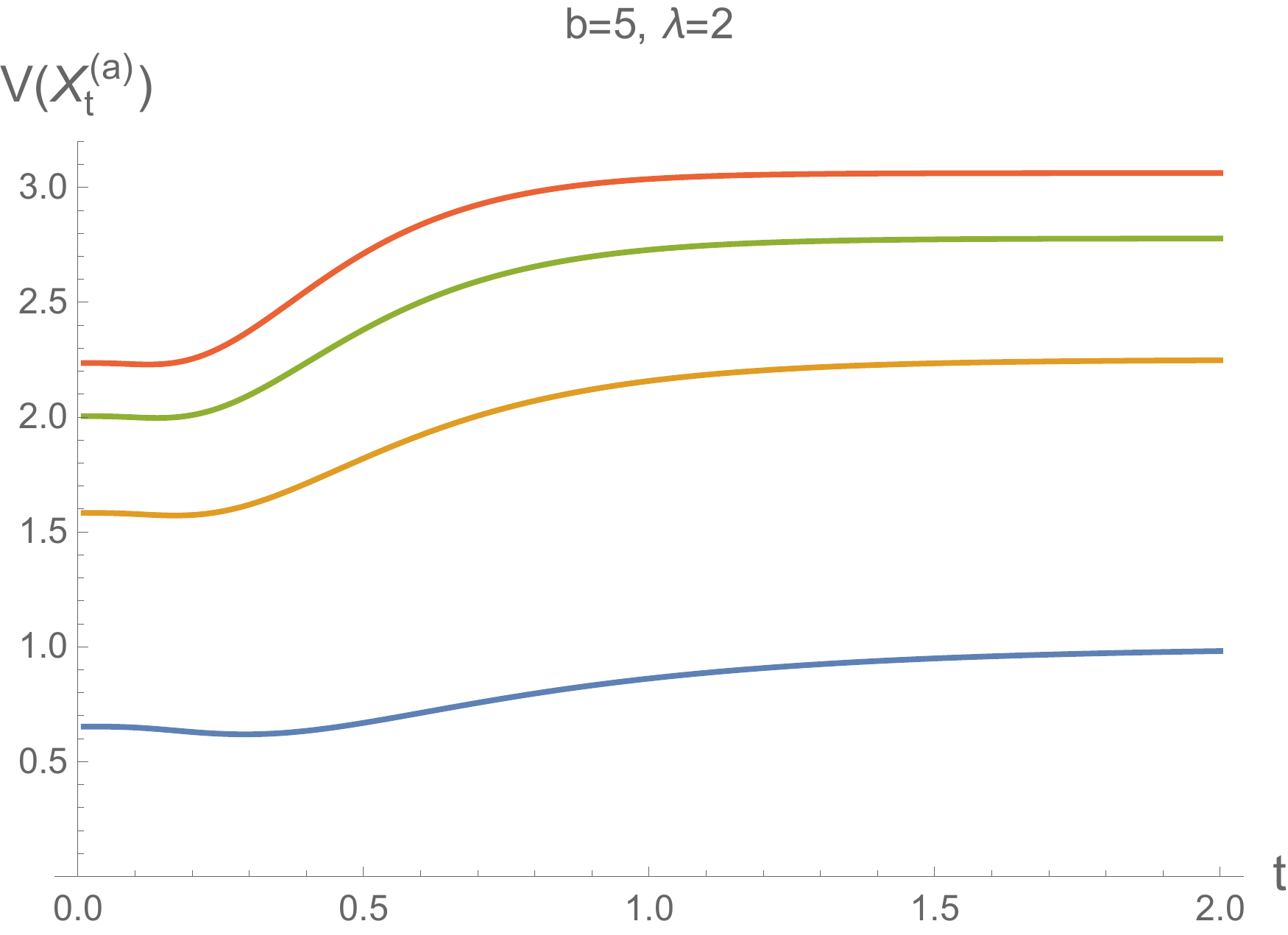}
%
\caption{The  residual varentropy of $X_t^{(a)}$ for the series system of Example \ref{ex:series}, 
for $a=n=1,2,3,4$ (from bottom to top) and for $b$ and $\lambda$ as indicated.}
\label{fig:PRH}
\end{figure}
%
\end{example}
\begin{example} \rm
Under the proportional hazards model, Eq.\ (\ref{eq:varentropy sistema serie}) 
can be used to construct time-varying reference sets for the information content of the residual lifetime (\ref{eq:Xt}). 
Specifically, we determine intervals of the form 
\begin{equation}
 H(X_t^{(a)})\pm k \sqrt{V(X_t^{(a)})} = \mathbb{E}[IC(X_t^{(a)})]\pm k \sqrt{ {\rm Var}[IC(X_t^{(a)})]},
 \qquad k=2,3
 \label{eq:interv}
\end{equation}
for suitable baseline distributions (Weibull, gamma and lognormal). Since closed forms are not available, 
we illustrate such results with some graphics given in Figure \ref{fig:interv}. For comparison purposes, 
the relevant parameters are chosen in order that the baseline distributions have unity means. 
%
\begin{figure}[t]  
\centering
\includegraphics[scale=0.35]{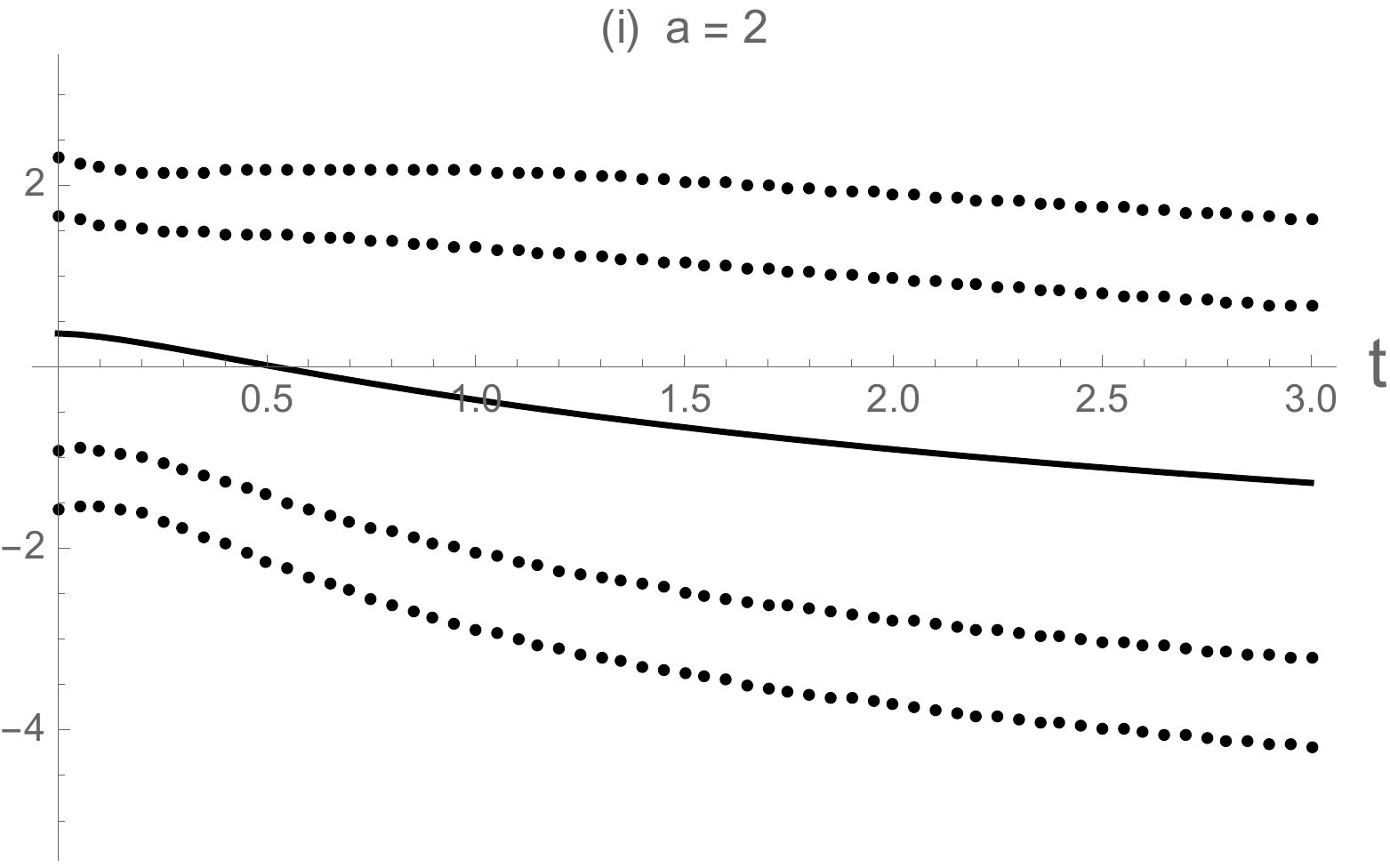}
$\,$
\includegraphics[scale=0.35]{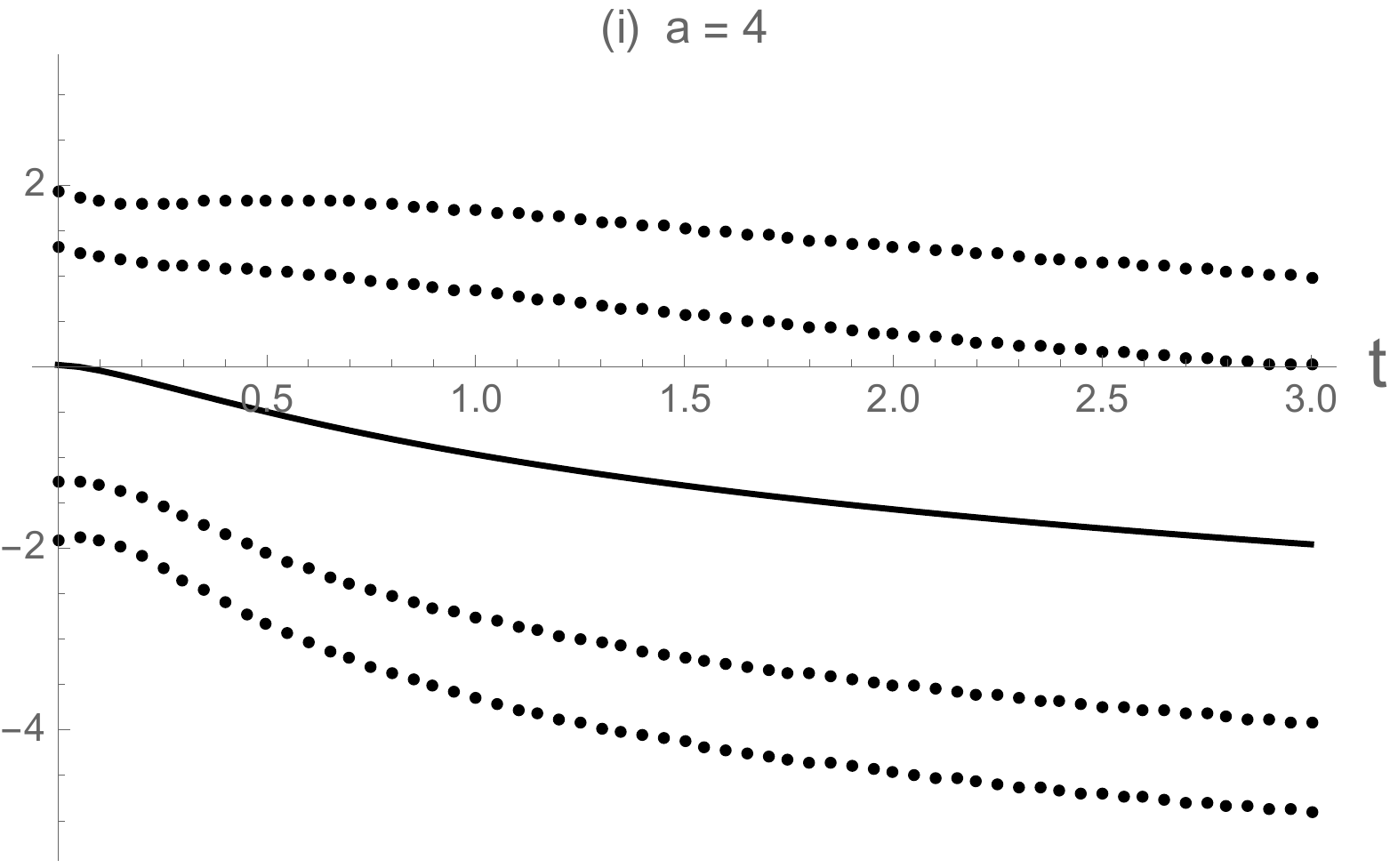}

\vspace{0.5cm}

\includegraphics[scale=0.35]{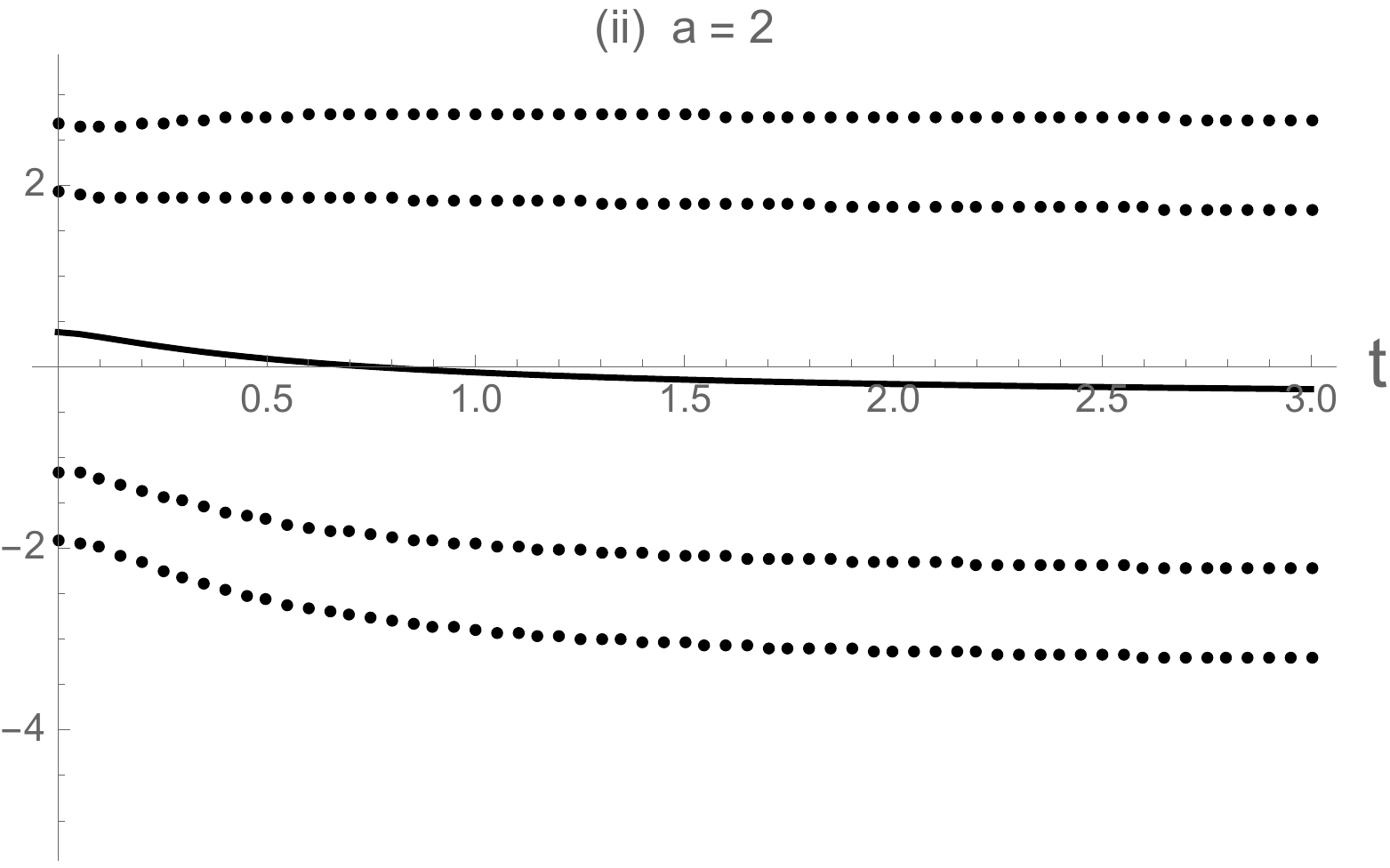}
$\,$
\includegraphics[scale=0.35]{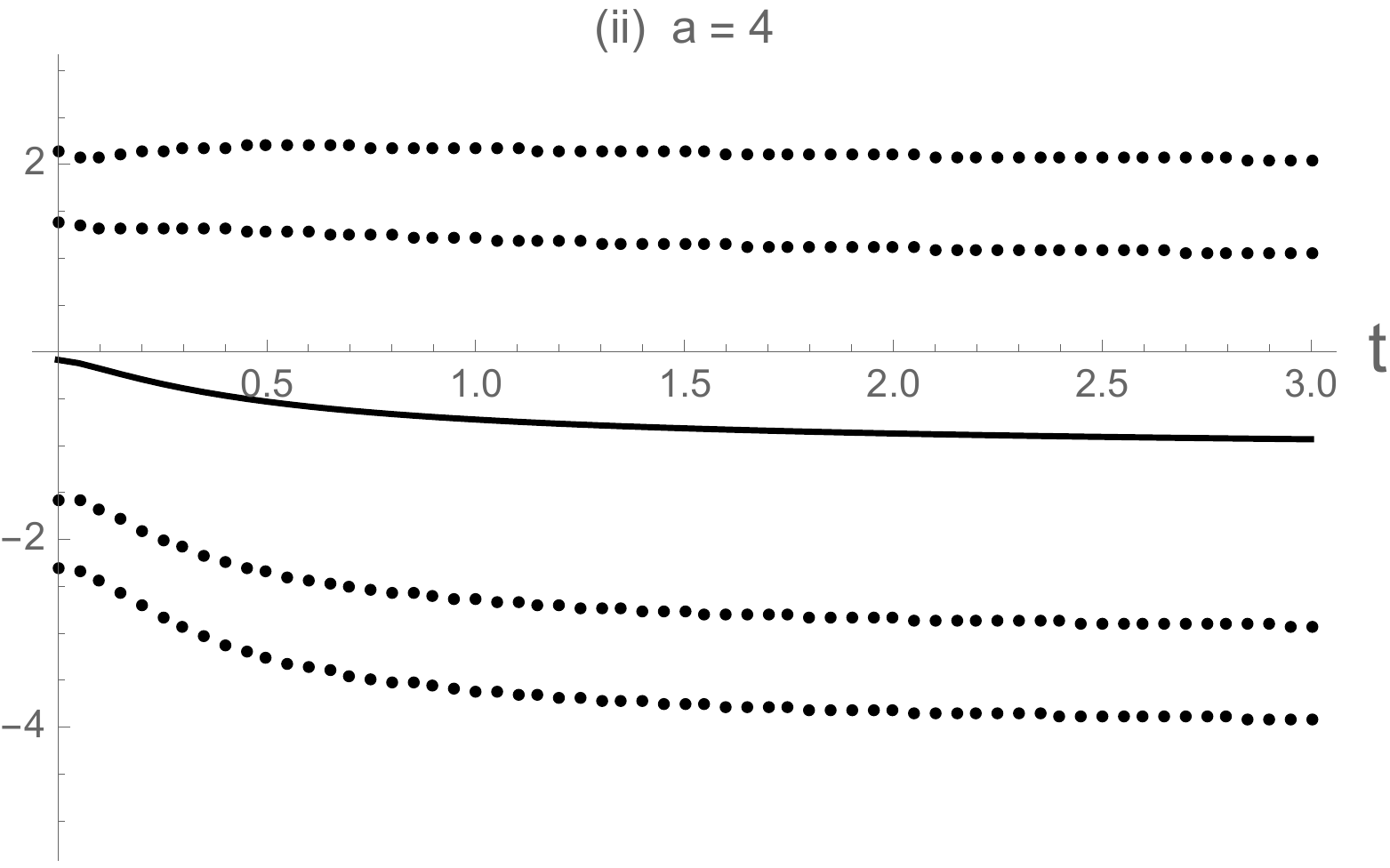}

\vspace{0.5cm}

\includegraphics[scale=0.35]{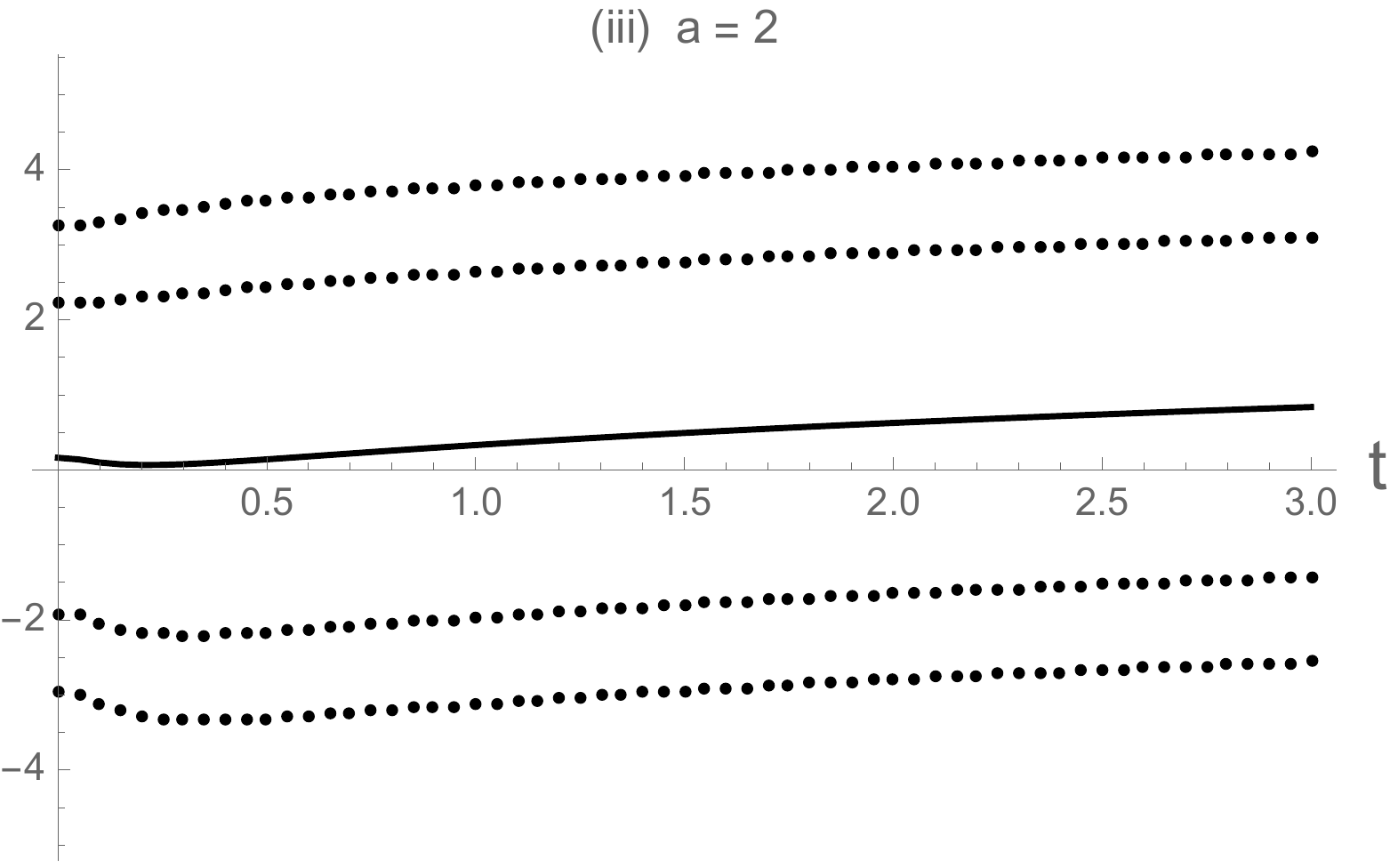}
$\,$
\includegraphics[scale=0.35]{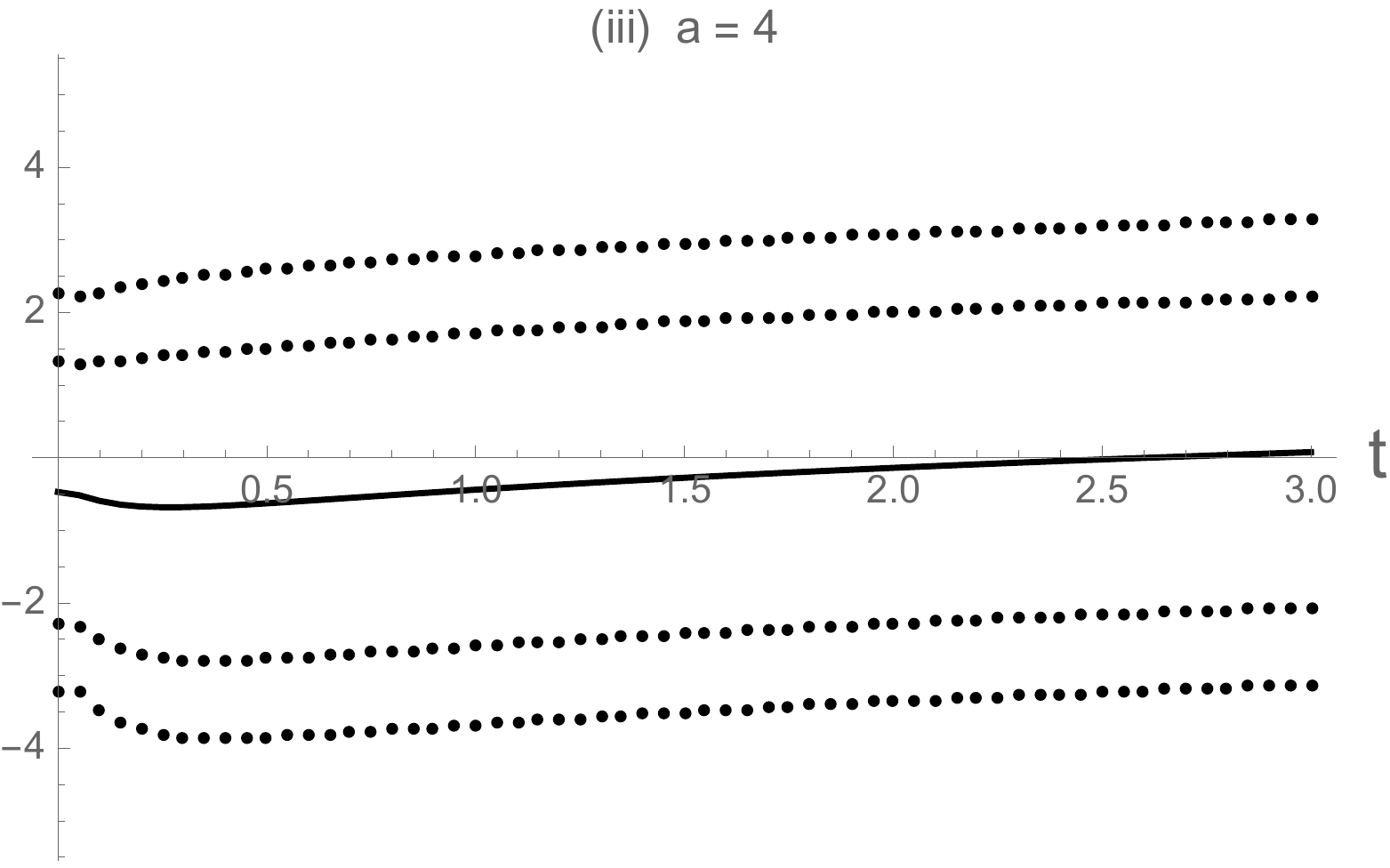}
\caption{
Residual entropy $H(X_t^{(a)})$ (full) and extremes of the intervals (\ref{eq:interv}) 
(dotted) with $a=2$ (left) and $a=4$ (right), for the following baseline pdfs: 
\\
(i) \ (Weibull)  $f(t)=\frac{k}{\lambda} \left(\frac{t}{\lambda}\right)^{k-1}\exp\{-\left(\frac{t}{\lambda}\right)^k\}$, $t> 0$, 
for $k = 2$, $\lambda=\frac{2}{\pi}$; 
\\
(ii) \ (gamma)  $f(t)=\frac{1}{\theta} \left(\frac{t}{\theta}\right)^{r-1}\exp\left\{-\frac{t}{\theta}\right\} \frac{1}{\Gamma(r)}$, 
$t> 0$, for $r = 2$, $\theta=\frac{1}{2}$; 
\\
(iii) \ (lognormal) $f(t)= \frac{1}{\sqrt {2\pi} \sigma t} \exp\left\{-\frac{(\log t-\mu )^{2}}{2\sigma^2}\right\}$, 
$t>0$,  for $\mu=-\frac{1}{2}$, $\sigma=1$.}
\label{fig:interv}
\end{figure}

\end{example}
%
\subsection{First-passage times of an Ornstein-Uhlenbeck jump-diffusion process}
The continuous-time Ehrenfest model describes a simple diffusion process as a suitable Markov chain, 
where molecules of a gas diffuse at random in a container divided into two equal parts by a permeable membrane. 
Recently, Dharmaraja {\em et al.}\ \cite{Dharmaraja} proposed an extension of such stochastic system that 
includes the occurrence of stochastic resets, also named `catastrophes', i.e.\ instantaneous transitions to the state 
zero at constant rate $\xi>0$. A jump-diffusion approximation was considered under a suitable scaling procedure.  
Specifically,  
the resulting jump-diffusion process, say $\{X(t), t\geq 0\}$, consists in a mean-reverting time-homogenous 
Ornstein-Uhlenbeck process with catastrophes (occurring with rate $\xi$), having state-space $\mathbb{R}$, 
with drift and infinitesimal variance given by 
$$
 A_1(x)=-\alpha x, \qquad A_2(x)=\alpha \nu \qquad (\alpha>0, \nu>0).
$$ 
In this case, denoting by $f(t)$ the first-passage-time (FPT) pdf of $X(t)$ through 0, with $X(0)=y\neq 0$, we have 
(cf.\ Eq.\ (49) of \cite{Dharmaraja})
\begin{equation} 
 f(t)=e^{-\xi t}\,\widetilde f(t)+\xi\,e^{-\xi t}{\rm Erf}\left(|y| e^{-\alpha t}\,
  [\nu(1-e^{-2\alpha t})]^{-1/2}\right),
 \qquad t> 0,
 \label{g_SC}
\end{equation}
with  $f(0)=\xi$, where ${\rm Erf}(\cdot)$ is the error function, and where (cf.\ Eq.\ (38) of \cite{Dharmaraja}) 
$$
 \widetilde f(t)={2\alpha|y|e^{-\alpha t}\over \sqrt {\pi\nu} \left(1-e^{-2\alpha t}\right)^{3/2}}
 \exp\biggl\{-\,{ y^2e^{-2\alpha t}\over \nu(1-e^{-2\alpha t})}\biggr\},
 \qquad  t > 0,
$$
with $ \widetilde f(0)=0$, is the FPT pdf of the corresponding diffusion process in absence of catastrophes.
We recall that  the FPT pdf (\ref{g_SC}) deserves interest in the realm of stochastic processes  
with stochastic reset (see, for instance, Kusmierz {\em et al}.\ \cite{Kusmierz2014} and Pal \cite{Pal2015}). 
To analyze the relevant information content, Figures \ref{fig:entropiaEsOUa} 
and \ref{fig:entropiaEsOUb} show some instances of the residual entropy related to pdf (\ref{g_SC}), 
whereas the corresponding residual varentropy is provided in Figures \ref{fig:ventropiaEsOUa} and \ref{fig:ventropiaEsOUb}. 
It is shown that the residual entropy is decreasing in $\xi$ and in $\nu$; moreover, it tends to a constant 
when $t$ grows, such limit being decreasing in $\xi$ and constant in $\nu$. 
The residual varentropy exhibits a different behavior, since it is decreasing in $\xi$ and is increasing in $\nu$ for 
sufficiently large values of $t$. Moreover, it tends to an identical limit when $t$ grows.  
This latter property is confirmed by extensive computations performed for various choices of the parameters.
%
\begin{figure}[t]  
\centering
\includegraphics[scale=0.45]{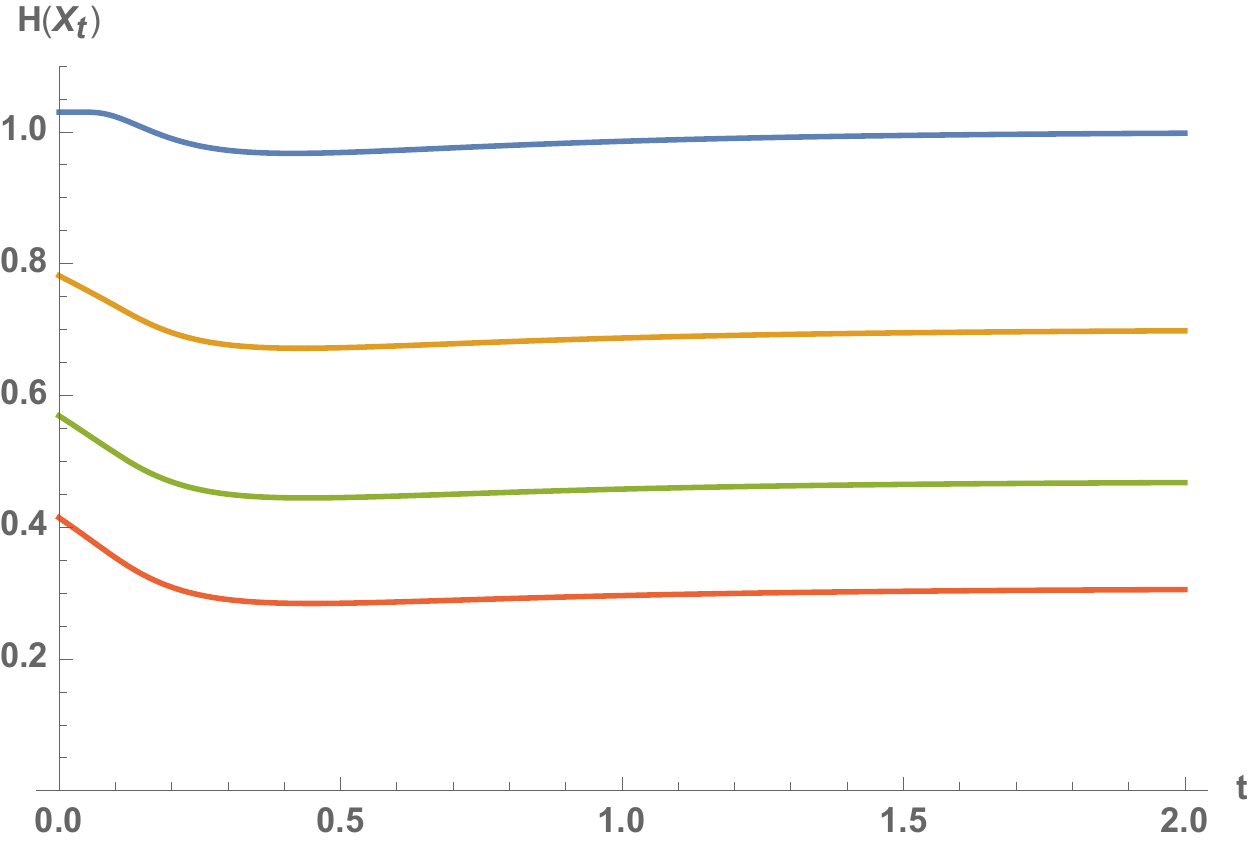}
$\;$
\includegraphics[scale=0.45]{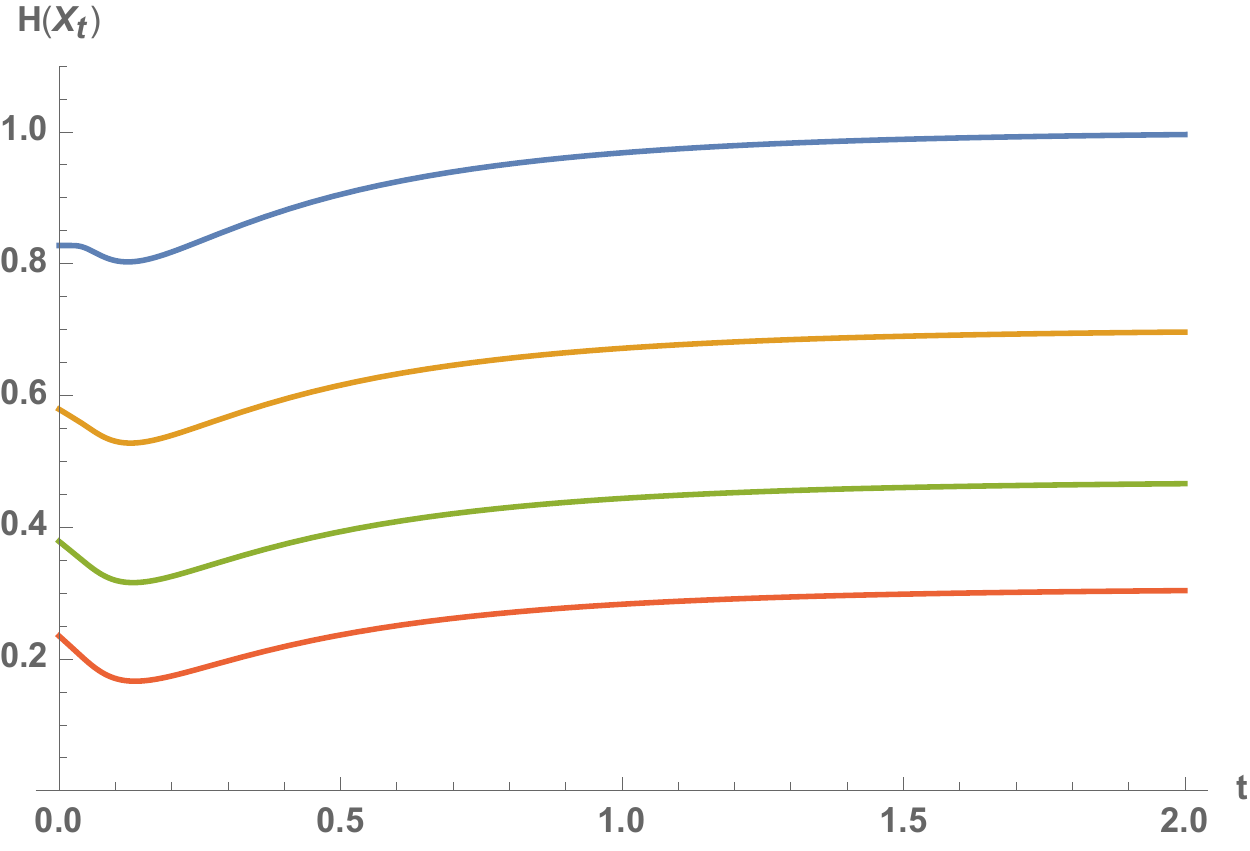}
\caption{Residual entropy for the FPT pdf (\ref{g_SC}), when $y=1$, $\alpha=1$, $\nu=1$ (left), $\nu=2$ (right), 
and $\xi=0$, $0.35$, $0.7$, $1$ (from top to bottom).}
\label{fig:entropiaEsOUa}
\end{figure}
\begin{figure}[t]  
\centering
\includegraphics[scale=0.45]{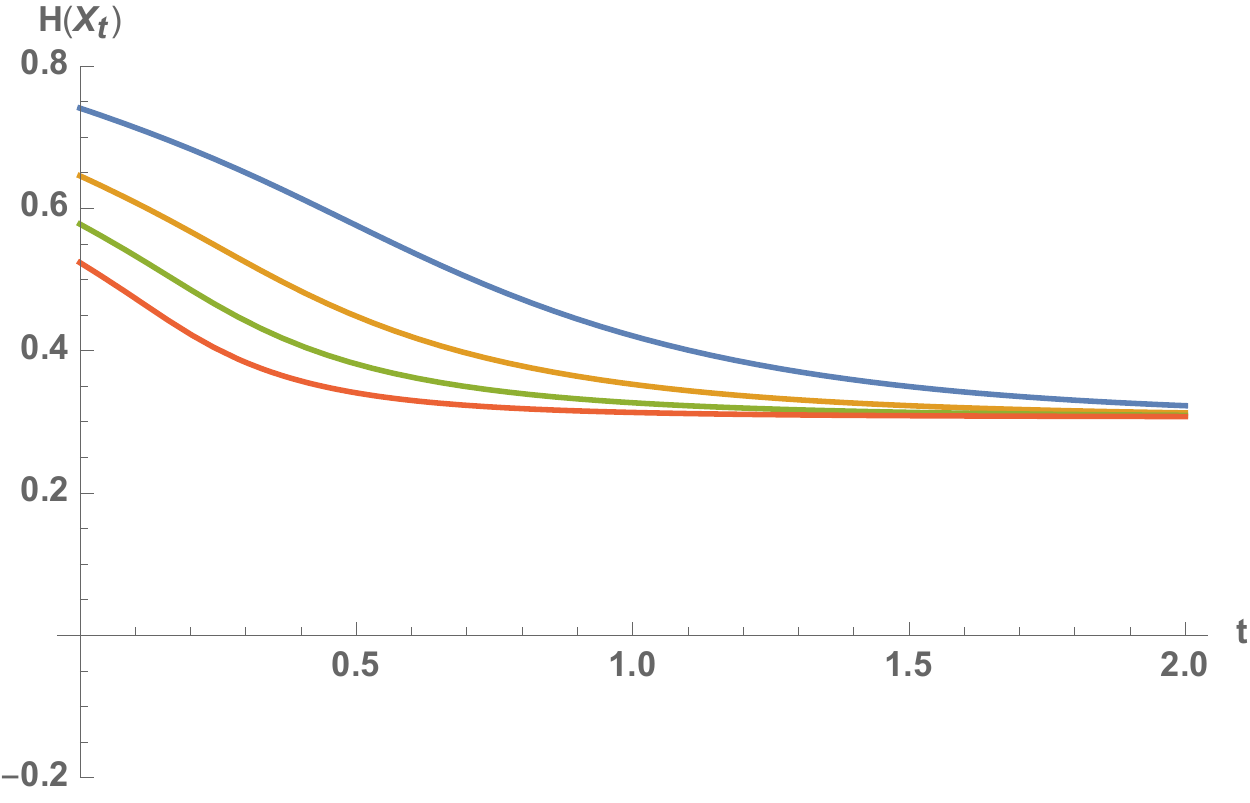}
$\;$
\includegraphics[scale=0.45]{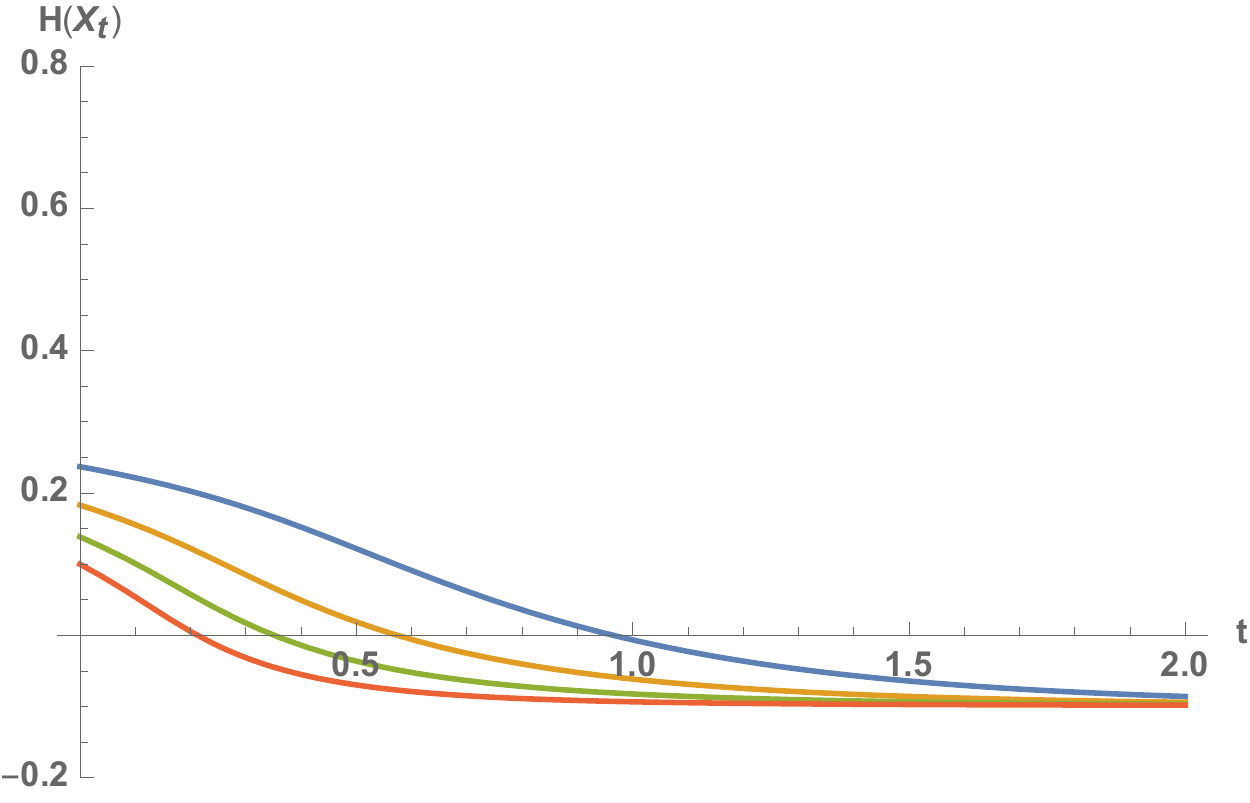}
\caption{Same as Figure \ref{fig:entropiaEsOUa}, for $\xi=1$ (left),  $\xi=2$ (right), 
and $\nu=0.15$, $0.3$, $0.45$, $0.6$ (from top to bottom).}
\label{fig:entropiaEsOUb}
\end{figure}
%
\begin{figure}[t]  
\centering
\includegraphics[scale=0.45]{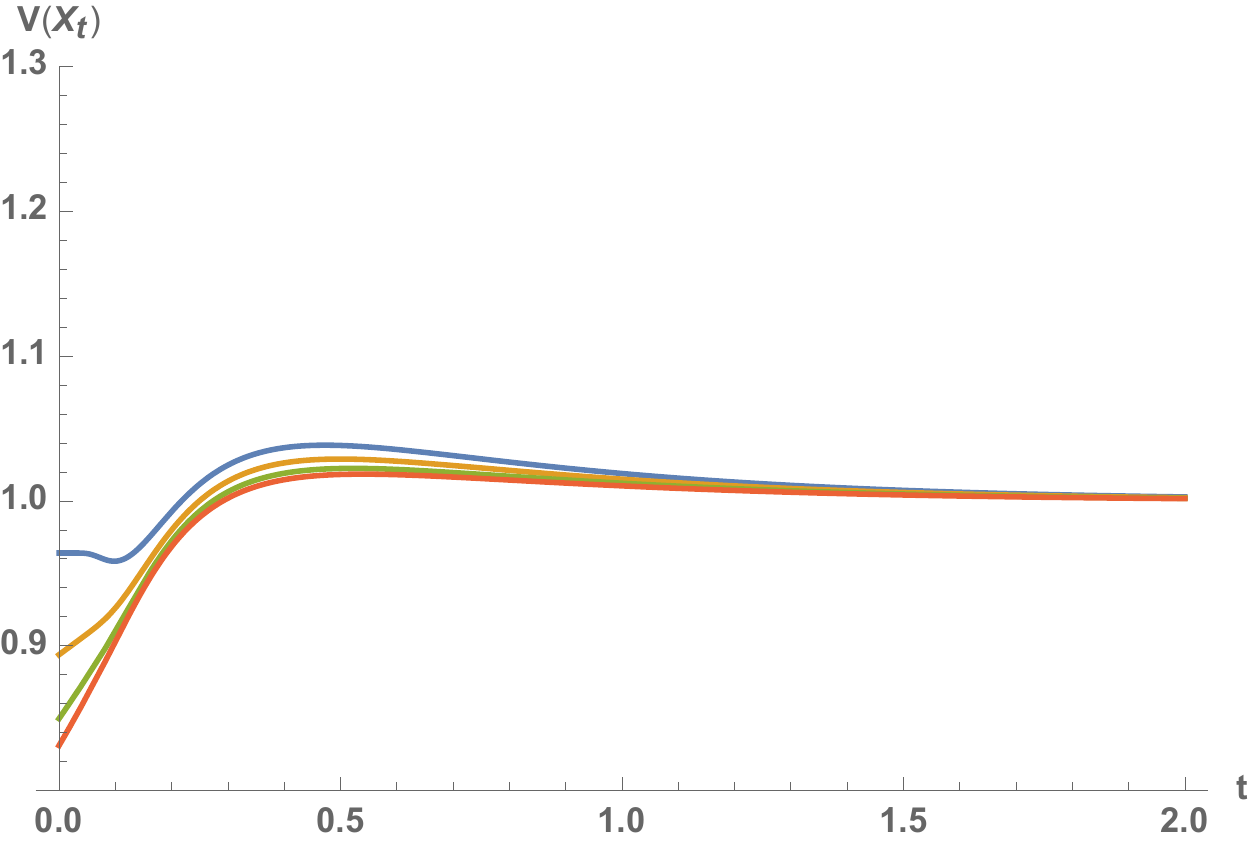}
$\;$
\includegraphics[scale=0.45]{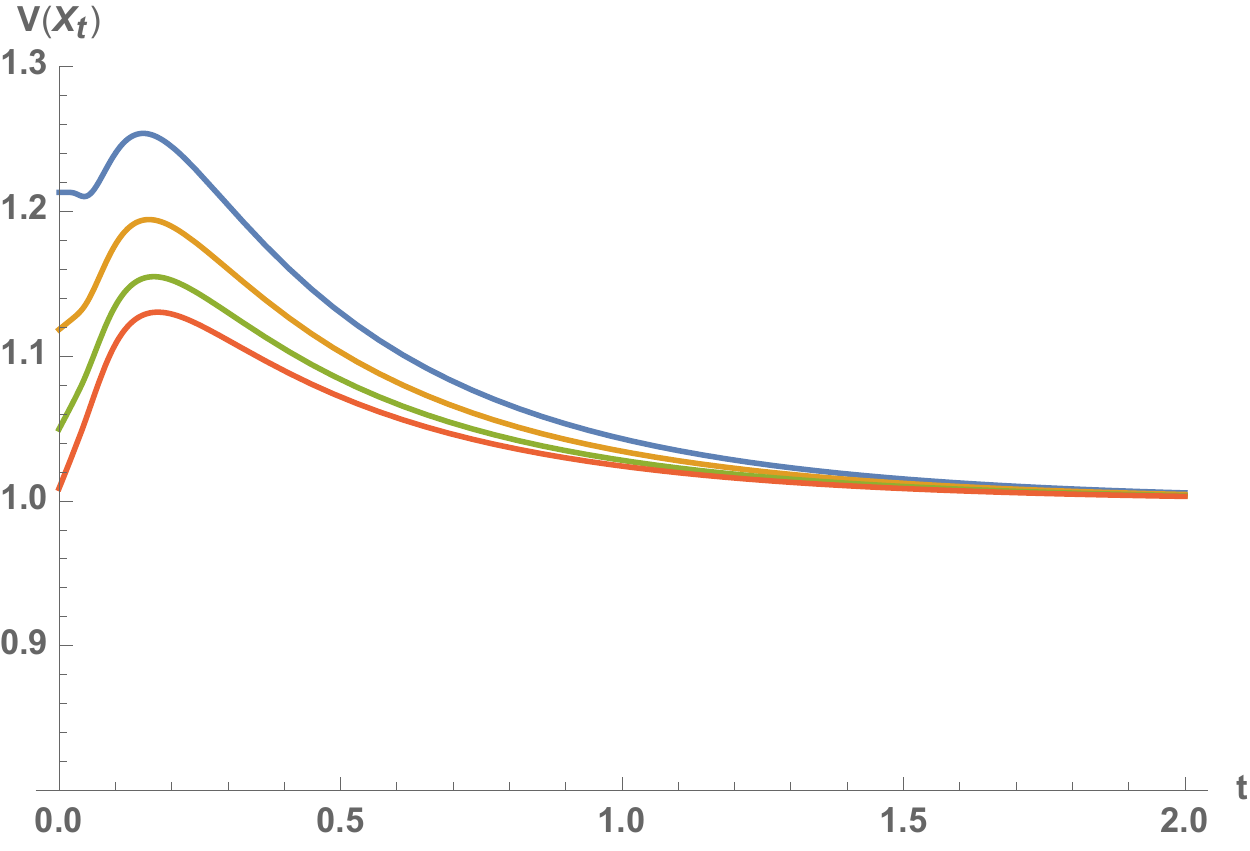}
\caption{Residual varentropy for the same cases of Figure \ref{fig:entropiaEsOUa}, with 
$\xi=0$, $0.35$, $0.7$, $1$ (from top to bottom).}
\label{fig:ventropiaEsOUa}
\end{figure}
\begin{figure}[t]  
\centering
\includegraphics[scale=0.45]{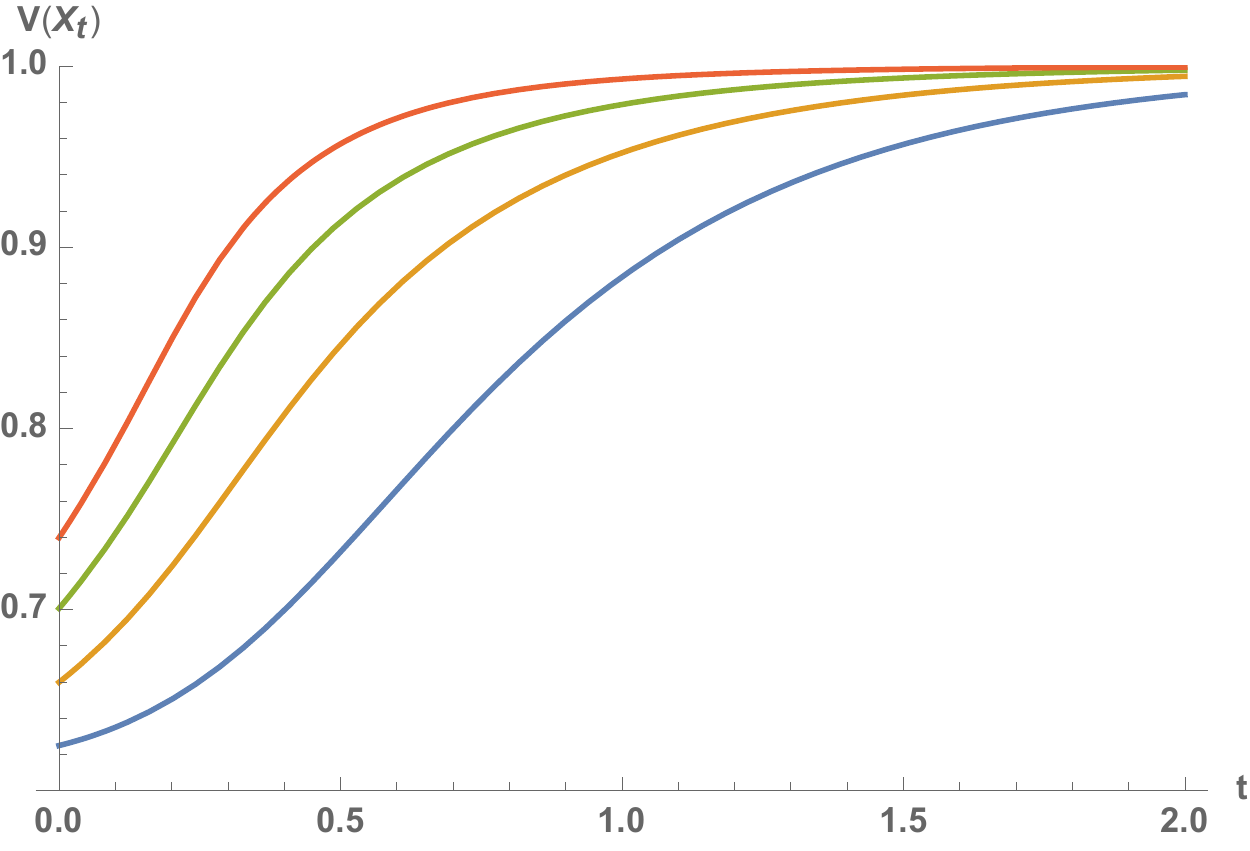}
$\;$
\includegraphics[scale=0.45]{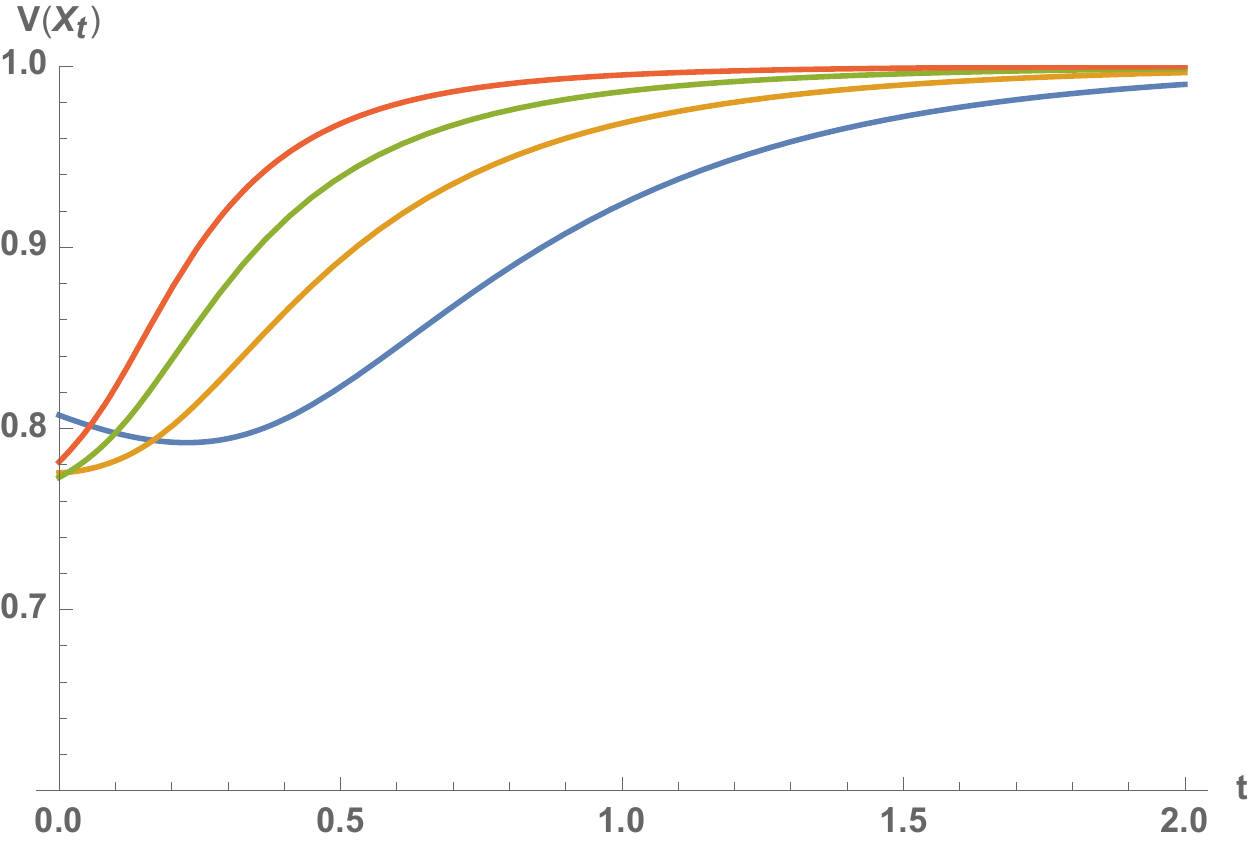}
\caption{Residual varentropy for the same cases of Figure \ref{fig:entropiaEsOUb}, 
with $\nu=0.15$, $0.3$, $0.45$, $0.6$ (from bottom to top).}
\label{fig:ventropiaEsOUb}
\end{figure}
%
\section{Conclusions}
The differential entropy (\ref{eq:shannonentropy}) is  largely used in information theory and other related areas, 
being the analogue of the Shannon entropy for a continuous random variable. It constitutes the expected value 
of the information content (\ref{eq:variabile informazione}), whereas its variance is given by the varentropy 
(\ref{eq:varentropy}). The latter is useful to assess the effectiveness of the differential entropy as a measure 
of the information content of a random system. 
\par
Motivated by possible application in reliability theory and survival analysis, in this paper we investigated 
the residual varentropy, that is the varentropy of the residual lifetime distribution. Together with the residual 
entropy, this measure allows to analyze the dynamical information content of time-varying systems 
conditional on being active at current time.   
We discussed various properties, with connections to the generalized hazard rate, the effect of linear 
transformations, and a suitable lower bound that involves the variance residual life function. 
We also addressed the use of the residual varentropy in connection with classical distributions 
and within some applications concerning the proportional hazards model and the first-passage time 
problem of an Ornstein-Uhlenbeck jump-diffusion process with catastrophes. 
\par
Future developments will be oriented to applications of the varentropy to other stochastic models 
of interest (such as  order statistics, spacings, record values,  
inaccuracy measures based on the relevation transform and its reversed version)  
and to construct an empirical version of the residual varentropy in order to come to suitable estimates. 
\subsection*{Acknowledgements}
%
The  authors are members of the research group GNCS of INdAM. (Istituto Nazionale di Alta Matematica). 
This research is partially supported by MIUR - PRIN 2017, project `Stochastic Models for Complex Systems', 
no.\ 2017JFFHSH.
%

%
\end{document}